\documentclass[review]{elsarticle}
\usepackage{geometry}
\usepackage{lineno,hyperref}
\usepackage{setspace} 

\usepackage{amsmath}
\usepackage{amssymb}
\usepackage{amsthm}
\usepackage{amsfonts}
\usepackage{amscd}
\usepackage{graphicx}
\usepackage{caption2}
\usepackage{float}
\usepackage{bm}
\usepackage{lscape}
\usepackage{mathrsfs}
\usepackage{verbatim}
\usepackage{multirow}
\usepackage{subfigure}
\usepackage{color}
\usepackage[ruled,vlined]{algorithm2e}

\newtheorem{theorem}{Theorem}[section]
\newtheorem{remark}[theorem]{Remark}
\newtheorem{lemma}[theorem]{Lemma}
\newtheorem{proposition}[theorem]{Proposition}
\newtheorem{definition}[theorem]{Definition}

\begin{document}

\begin{frontmatter}
\title{Convergence analysis of Galerkin finite element approximations to shape gradients in eigenvalue optimization\tnoteref{mytitlenote}}
\tnotetext[mytitlenote]{This work was supported in part by Science and Technology Commission of Shanghai Municipality (No. 18dz2271000) and the National Natural Science Foundation of China under grants 11201153 and 11571115.}

\author[sfzhu]{Shengfeng Zhu\corref{cor1}}
\address[sfzhu]{Department of Data Mathematics \& Shanghai Key Laboratory of Pure Mathematics and Mathematical Practice, School of Mathematical Sciences, East China Normal University, Shanghai 200241, China}
\ead{sfzhu@math.ecnu.edu.cn}
\cortext[cor1]{Corresponding author}

\author[xlhu]{Xianliang Hu}
\address[xlhu]{School of Mathematical Sciences, Zhejiang University, Hangzhou 310027, Zhejiang, China}
\ead{xlhu@zju.edu.cn}

\begin{abstract}
Numerical computation of shape gradients from Eulerian derivatives is essential to wildly used gradient type methods in shape optimization. Boundary type Eulerian derivatives are popularly used in literature. The volume type Eulerian derivatives hold more generally, but are rarely noticed and used numerically. We investigate thoroughly the accuracy of Galerkin finite element approximations of the two type shape gradients for optimization of elliptic eigenvalues. Under certain regularity assumptions on domains, we show \emph{a priori} error estimates for the two approximate shape gradients. The convergence analysis shows that the volume integral formula converges faster and generally offers better accuracy. Numerical experiments verify theoretical results for the Dirichlet case. For the Neumann case, however, the boundary formulation surprisingly converges as fast as the volume one. Numerical results are presented.
\end{abstract}


\end{frontmatter}

\section{Introduction}
With the development of computer sciences, shape optimization has become important and promising in many fields of engineering (cf \cite{Defour,HN,HP,Pironneau,Sok1}), e.g., structural mechanics \cite{Bendsoe}, acoustics \cite{Osher2}, computational fluid dynamics (see e.g., \cite{JHH,Pironneau,QR}), etc. Analytic approaches can be applied to study existence and regularity \cite{BB,Defour}. For shape design of complex systems in practice, numerical methods such as gradient-type algorithms are usually employed instead to seek ``approximately'' optimal shapes. One can adopt two strategies: \emph{optimize-then-discretize} \cite{HN,Sok1} and \emph{discretize-then-optimize} \cite{Bendsoe}. They are not equivalent at certain circumstances. For the former, the so-called \emph{Eulerian derivative} and corresponding \emph{shape gradient} of a \emph{shape functional} are usually required by \emph{shape sensitivity analysis} for sensitivity calculation of shape functionals with respect to domain variations (see e.g., \cite{Defour,Sok1}). In 1907, Hadamard computed the Eulerian derivative of the first eigenvalue of a clamped plate with $C^\infty$ smooth boundary \cite{Hadamard}. Later, a \emph{structure theorem} was developed by Zol\'{e}sio for general shape functionals on $C^{\zeta+1}$-domains ($\zeta\geq 0$). By the structure theorem, the Eulerian derivative can be expressed as a boundary integral. Due to the attractive concise appearance, this type Eulerian derivative has caused much attention both in shape optimization theory \cite{Defour,Sok1} and most existing numerical shape gradient algorithms (see e.g. \cite{Pironneau}). But this type Eulerian derivative actually fails to hold when the boundary does not satisfy the required smoothness. Another more general type Eulerian derivative expressed as a domain integral then can be used instead \cite{Defour}. The two type Eulerian derivatives are equivalent through integration by parts if the boundary is regular enough.

For shape gradient computations, numerical approaches such as finite elements, finite differences (see e.g. \cite{ZhuJCP}) and boundary element methods \cite{AF} are used to solve the state and possible adjoint constraints. Galerkin finite element method (cf. \cite{Brenner}) is popular for discretizations of PDEs in shape optimization (see e.g. \cite{HN,LiuZhu,WHZ,ZhuANM,ZHW}). This method based on domain triangulation is flexible to shape representation and shape changes in shape optimization.

The accuracy of finite element approximations of shape gradients seems to be essential for implementation of numerical optimization algorithms. Delfour and Zol\'{e}sio (Remark 2.3 on pp. 531 \cite{Defour}) pointed that ``\emph{the boundary integral expression is not suitable since the finite element solution does not have the appropriate smoothness under which the boundary integral formula is obtained}''. Pironneau et al. (pp 210 \cite{MP}) presented convergence analysis for consistent approximations of boundary shape gradients in linear elliptic problems. Bergren \cite{Berggren} remarked that ``\emph{the sensitivity information-directional derivatives of objective functions and constraints needs to be very accurately computed in order for the optimization algorithms to fully converge}''. The use of domain expressions of Eulerian derivatives seems to be more promising. Recently, Hiptmair et al. \cite{Hiptmair} first showed that the Galerkin finite element approximation of shape gradient in the volume integral type converges faster and is more accurate than that in the boundary integral for linear elliptic problems. The volume formulations of shape gradients are derived and used for numerical shape optimization algorithms in magnetic induction tomography \cite{HLY} and parabolic diffusion problems \cite{SSW}, respectively. Shape gradients are popular in boundary form when combined with the level set method for shape optimization (see e.g. \cite{Allaire,Burger,LDZKL,LiuZhu,Osher2}). Recently, the volume type Eulerian derivatives were also incorporated numerically into the level set method for shape and topology optimization \cite{Burman,Laurain}. In \cite{Kiniger}, an eigenvalue shape optimization problem was transformed to be an optimal control problem and a priori error estimates were obtained after finite element discretizations.

In this paper, we prove convergence for Galerkin finite element approximations of shape gradients in eigenvalue optimization. The motivations arise from the following aspects. First, eigenvalue problems in optimal shape design have fundamental importance for science and engineering, especially in structural mechanics (see. e.g., \cite{Allaire2,Bendsoe,Defour,HenriBook,Rousselet83,Sok1,ZhuJCP}). Second, finite element approximations to shape gradients of eigenvalues in boundary formulations are wildly used for numerical algorithms in eigenvalue optimization (see e.g., \cite{Allaire2,AF,Defour,HenriBook,Osher2,Oudet,Sok1} and references therein). Third, numerical evidence shows that the potential advantages for using the \emph{new} volume shape gradients in algorithms for shape optimization of Laplace eigenvalue problems \cite{ZhuJOTA}. To the best of our knowledge, there is no literature reported on convergence analysis of the approximate shape gradients in boundary/volume formulation for eigenvalue problems. For both Dirichlet and Neumann Laplace eigenvalue problems \cite{AF,BB,HenriBook,KS,LNS,Oudet}, we prove convergence of the finite element approximations to shape gradients in both boundary and volume formulations. \emph{A priori} error estimates are presented first in an infinite-dimensional operator norm. Numerical results are presented for verifying convergence of approximate shape gradients.

The rest of the paper is organized as follows. Laplace eigenvalue problems with Dirichlet and Neumann boundary conditions are presented in Section 2. Shape calculus is briefly introduced to give the Eulerian derivatives of eigenvalues in the forms of boundary and volume integrals. In Section 3, under certain regularity assumptions on domains, we present \emph{a priori} convergence analysis of the finite element approximations to shape gradients in boundary and volume formulations. In Section 4, numerical results are presented for verifying convergence of approximate shape gradients as well as effectiveness of shape gradient algorithms in shape optimization. Brief conclusions are drawn in Section 5.

\section{Problem formulation}
Let $\Omega$ be a bounded domain in $\mathbb{R}^d$ $(d=2,3)$ with Lipschitz continuous boundary $\partial\Omega$. We consider the Laplace eigenvalue problem:
\begin{equation}\label{EigModel}
\left\{
\begin{aligned}
&-\Delta u =\lambda u  &&{\rm in}\ \Omega\\
&u=0\quad {\rm or}\quad \frac{\partial u}{\partial n}=0\  &&{\rm on} \ \partial \Omega,
\end{aligned}
\right.
\end{equation}
where $\Delta=\sum_{i=1}^d \partial^2/\partial x_i^2$ is the Laplacian. The homogeneous Dirichlet or Neumann boundary condition physically corresponds to a vibrating planar membrane ($d=2$) being fixed or free, respectively.

Let us first introduce briefly some notations on Sobolev spaces \cite{Adams}. For $1\leq p\leq \infty$, the Banach space $L^p(\Omega)$ consists of measurable functions $v$ such that the associated norm $\|v\|_{L^p(\Omega)}<\infty$.
For each integer $m\geq 0$, the Banach space
is equipped with the norm $\|\cdot\|_{W^{m,p}(\Omega)}$.
In particular, when $p=2$, we write $H^m(\Omega)$ instead of $W^{m,2}(\Omega)$ and $\|\cdot\|_{H^m(\Omega)}$ instead of $\|\cdot\|_{W^{m,2}(\Omega)}$. Notice that $H^m(\Omega)$ is indeed a Hilbert space with respect to the scalar product
$(w,v)_{H^m(\Omega)}:=\sum_{|\alpha|\leq m}(D^\alpha w,D^\alpha v)_{L^2(\Omega)}$
with $(\cdot,\cdot)_{L^2(\Omega)}$ being the usual $L^2$ inner product. Denote by $W^{m,p}_0(\Omega)$ the closure of $C^\infty_0(\Omega)$ with respect to the norm $\|\cdot\|_{W^{m,p}(\Omega)}$.
We write $H^m_0(\Omega)$ instead of $W^{m,2}_0(\Omega)$ when $p=2$. For readability, we use the same notations for Sobolev norms of vector-valued and scalar functions.

The variational formulation of (\ref{EigModel}) is to find $\lambda\in \mathbb{R}, 0\neq u\in V $ such that \cite{Babuska}
\begin{equation}\label{varForm}
(\nabla u,\nabla v) =\lambda(u,v)\quad \forall v\in V,
\end{equation}
where $V=H^1_0(\Omega)$ ($V=\{v\in H^1(\Omega)|\int_\Omega v {\rm d}x = 0\}$) for the Dirichlet (Neumann) boundary condition. Due to the positiveness, self-adjointness and the compactness of the inverse of negative Laplacian operator, there exists a sequence of eigenpairs $\{(\lambda_i,u_i)\}_{i=0}^\infty$ as solutions of (\ref{varForm}) with eigenvalues
\begin{equation}\label{order}
0<\lambda_1\leq\lambda_2\leq\cdots\nearrow+\infty
\end{equation}
and corresponding eigenfunctions
$u_1,u_2,\cdots$,
which can be normalized with 
\begin{equation}\label{contNorm}
(u_i,u_j)=\delta_{ij},
\end{equation}
where $\delta_{ij}$ is the Kronecker delta. Note that for the case of Neumann boundary condition, $\lambda_1\geq 0$. A typical optimization problem consists of minimizing some eigenvalue of (\ref{order}) subject to a volume constraint \cite{AF,BB,HenriBook,KS,Oudet}.


\subsection{{\emph A priori} error estimates for Laplace eigenvalue problems}\label{FEM}
We consider the standard Ritz-Galerkin finite element method \cite{Brenner} for discretization and approximation of the variational formulation (\ref{varForm}) \cite{Boffi}. For the shape gradient deformation algorithm we shall present, the domain $\Omega$ here at each iteration is naturally assumed to be a polygon/polyhedron, which can be triangulated exactly with no geometric error introduced.
\begin{remark}\label{regulartity}
For Dirichlet Laplace eigenvalue problems on planar polygonals (see Remark 4.2 \cite{Babuska}), we have the following regularity result similar as linear elliptic problems \cite{Grisvard}. Let $\omega\pi$ ($0<\omega\leq 2$) be the maximal interior angle of the vertices of $\Omega$. Then, we have the for a Laplace eigenfunction $u$ such that
\begin{equation*}
u\in H^{r}(\Omega)\cap H^1_0(\Omega)\quad {\rm with}\quad 1< r<1+\frac{1}{\omega},
\end{equation*}
i.e., $r$ can be $1+\frac{1}{\omega}-\epsilon$ for any small $\epsilon>0$. If $\Omega$ is a convex polygon, then $u\in H^2(\Omega)$. For $d=3$, more delicated discussions required to be made for the Poisson equation and we assume that $u\in H^r(\Omega)\cap H^1_0(\Omega)$ ($r>1$) for simplicity.
\end{remark}

Consider a family of triangulations $\{\mathcal{T}_h\}_{h>0}$ satisfying that $\overline{\Omega}=\cup_{K\in\mathcal{T}_h}\overline{K}$, where the mesh size $h:=\max_{K\in\mathcal{T}_h}{h_K}$ with $h_K:={\rm diam}\{K\}$ for any $K\in\mathcal{T}_h$. Let $\{V_h\}_{h>0}$ be a family of finite-dimensional subspaces of $H^1_0(\Omega)$. For the linear Lagrange elements, $V_h:=\{v_h\in C^0(\overline{\Omega})\cap H^1_0(\Omega):v_h|_K\in \mathbb{P}_1(K)\ \forall K\in \mathcal{T}_h\}$ in the Dirichlet case with $\mathbb{P}_1(K)$ denoting the set of piecewise linear polynomials on $K$ and $V_h=\{v_h\in C^0(\overline{\Omega}):v_h|_K\in \mathbb{P}_1(K)\ \forall K\in \mathcal{T}_h,\int_\Omega v_h {\rm d}x = 0\}$ in the Neumann case. Denote $r:=1+s$ with $0< s\leq 1$. Throughout, we shall denote by $C$ a general constant, which may differ at different occurrences and may depend on  the mesh aspect ratio and the shape of $\Omega$, but it is always independent of eigenfunction and the mesh size $h$. We assume that the mesh family $\{\mathcal{T}_h\}_{h>0}$ is \emph{regular} so that the following approximation property holds \cite{Brenner}:
\begin{equation}\label{approximation}
\inf_{v_h\in V_h}(\|u-v_h\|_{L^2(\Omega)}+h\|\nabla u-\nabla v_h\|_{L^2(\Omega)} ) \leq C h^{r}|u|_{H^r(\Omega)}\quad \forall u\in H^r(\Omega).
\end{equation}
 Suppose moreover that the mesh is \emph{quasi-uniform}, i.e.,
\begin{equation*}
\min_{K\in\mathcal{T}_h} h_K\geq C h\quad \forall h>0,
\end{equation*}
based on which the \emph{inverse inequality} holds (see e.g. Theorem 4.5.11 \cite{Brenner}). The weak formulation for conforming finite element approximation of the problem (\ref{varForm}) reads: find $\lambda_h\in\mathbb{R}$ and $0\neq u_h\in V_h $ such that
\begin{equation}\label{FEMvarForm}
(\nabla u_h,\nabla v_h) =\lambda_h(u_h, v_h) \quad \forall v_h\in V_h,
\end{equation}
For (\ref{FEMvarForm}), there exist a finite sequence of eigenvalues
\begin{equation*}
0<\lambda_{1,h}\leq\lambda_{2,h}\leq\cdots\leq \lambda_{N,h},\quad N={\rm dim}\ V_h,
\end{equation*}
and corresponding eigenvectors
\begin{equation*}
u_{1,h},u_{2,h}\cdots u_{N,h},
\end{equation*}
which can be assumed to satisfy
\begin{equation}\label{discNorm}
(u_{i,h}, u_{j,h})=\delta_{ij}.
\end{equation}
For $i=1,2,\cdots$, we suppose that $k_i$ is the lowest index of the $i$th distinct eigenvalue of (\ref{varForm}) with $l_i$ being its multiplicity. More precisely, suppose that
\begin{equation*}
\lambda_{k_{i-1}+l_{i-1}-1} = \lambda_{k_{i}-1}< \lambda_{k_{i}}=\lambda_{k_{i}+1}=\cdots = \lambda_{k_{i}+l_i-1}<\lambda_{k_{i}+l_i}=\lambda_{k_{i+1}}.
\end{equation*}
We have the following \emph{a priori} error estimates on approximate eigenvalues and eigenfunctions.
\begin{lemma}\label{basicErrEst}
Assume that $\Omega$ is a polygon/polyhedron and $\{\mathcal{T}_h\}_{h>0}$ are quasi-uniform. Let $(\lambda_{k_i+j-1},u_{k_i+j-1})$ and $(\lambda_{k_i+j-1,h},u_{k_i+j-1,h})$ be an eigenpair of (\ref{varForm}) and (\ref{FEMvarForm}), respectively with $u_{k_i+j-1}\in H^{1+s}(\Omega)$ ($0<s\leq 1$). Then,
\begin{equation*}\label{lambdaEst}
\lambda_{k_i+j-1} \leq \lambda_{k_i+j-1,h}\leq \lambda_{k_i+j-1}+ Ch^{2s}|u_{k_i+j-1}|_{H^{1+s}(\Omega)},
\end{equation*}
and $u_1,u_2,\cdots$ can be chosen so that (\ref{discNorm}) holds and
\begin{equation*}\label{uErr}
\|u_{k_i+j-1}- u_{k_i+j-1,h}\|_{L^2(\Omega)}+h\|\nabla u_{k_i+j-1}- \nabla u_{k_i+j-1,h}\|_{L^2(\Omega)}\leq Ch^{1+s}|u_{k_i+j-1}|_{H^{1+s}(\Omega)},
\end{equation*}
\begin{equation*}\label{uuee}
\|\nabla u_{k_i+j-1}- \nabla u_{k_i+j-1,h}\|_{H^{-s}(\Omega)}\leq Ch^{2s}|u_{k_i+j-1}|_{H^{1+s}(\Omega)},
\end{equation*}
where $j=1,2,\cdots,q_i$ with $i=1,2,\cdots$.
\end{lemma}
\begin{proof}
By combining (3.25) \cite{KO} and Theorem 5.1 \cite{BO1989}, we obtain the first two \emph{a priori} error estimates above. To prove the last inequality, we first have
\begin{equation}\label{wefwfefw}
\begin{aligned}
  \|\nabla u -\nabla u_h\|_{H^{-1}(\Omega)}&=\sup_{0\neq v\in H_0^{1}(\Omega)^d}\frac{\langle \nabla u-\nabla u_h, v \rangle}{\|v\|_{H^{1}(\Omega)}}\\
  &=\sup_{0\neq v\in H_0^{1}(\Omega)^d}\frac{(u-u_h, {\rm div} v)}{\|v\|_{H^{1}(\Omega)}}\\
  &\leq \|u-u_h\|_{L^2(\Omega)}\sup_{0\neq v\in H_0^{1}(\Omega)^d}\frac{\|{\rm div} v\|_{L^2(\Omega)}}{\|v\|_{H^{1}(\Omega)}}\\
  &\leq C\|u-u_h\|_{L^2(\Omega)}\\
  &\leq Ch^2|u|_{H^2(\Omega)}\quad {\rm for}\ u\in H^2(\Omega),
\end{aligned}
\end{equation}
where we have omitted the subscript $k_i+j-1$ for notational simplicity.
On the other hand,
\begin{equation}\label{fewfw}
  \|\nabla u- \nabla u_h\|_{L^2(\Omega)}\leq C |u|_{H^1(\Omega)}\quad {\rm for}\ u\in H^1(\Omega).
\end{equation}
Using the operator interpolation theorem (Lemma 22.3 \cite{Tartar}) to (\ref{wefwfefw}) and (\ref{fewfw}), we have
\begin{equation*}
  \|\nabla u- \nabla u_h\|_{H^{-s}(\Omega)}\leq C h^{2s}|u|_{H^{1+s}(\Omega)}\quad {\rm for}\ u\in H^{1+s}(\Omega), \quad 0\leq s\leq 1.
\end{equation*}
\end{proof}
Note that this Lemma has included results for the special case of simple eigenvalues, i.e., $l_i=1$. In following, we omit the index number for a specific eigenvalue and eigenfunction for simplicity. Let $(\lambda,u)$ and $(\lambda_{h},u_{h})$ be eigenpairs of (\ref{varForm}) and (\ref{FEMvarForm}), respectively.

Let us first define the Ritz projection $P_h: H^1_0(\Omega)\rightarrow V_h$ such that
\begin{equation}\label{RitzDef}
(\nabla P_h u, \nabla v_h)= (\nabla u,\nabla v_h)\quad \forall v_h\in V_h.
\end{equation}
\begin{lemma}\label{basicErrEst1}
Let assumptions in Lemma \ref{basicErrEst} hold with $s=1$. Then,
\begin{equation*}
\|\nabla(P_hu-u_h)\|_{L^2(\Omega)} \leq C h^{2}|u|_{H^{2}(\Omega)}.
\end{equation*}
\end{lemma}
\begin{proof}
We take $v_h=P_h u-u_h$ in(\ref{varForm}), (\ref{FEMvarForm}) and (\ref{RitzDef}). Then, we have
\begin{equation*}
\begin{aligned}
(\nabla(P_hu-u_h),\nabla(P_hu-u_h))
&= (\lambda u-\lambda_hu_h,P_hu-u_h)\\
&= \big(\lambda(u-u_h)+(\lambda-\lambda_h)u_h,P_hu-u_h\big).
\end{aligned}
\end{equation*}
Then, by the Cauchy-Schwarz inequality and triangle inequality,
\begin{equation*}
\begin{aligned}
\|\nabla(P_hu-u_h)\|_{L^2(\Omega)}^2 &\leq \big(\lambda \|u-u_h\|_{L^2(\Omega)}+|\lambda-\lambda_h|\|u_h\|_{L^2(\Omega)}\big)\|P_hu-u_h\|_{L^2(\Omega)}\\
&=\big(\lambda \|u-u_h\|_{L^2(\Omega)}+|\lambda-\lambda_h|\big)\|P_hu-u_h\|_{L^2(\Omega)}\\
&\leq C\big(\lambda \|u-u_h\|_{L^2(\Omega)}+|\lambda-\lambda_h|\big)\|\nabla(P_hu-u_h)\|_{L^2(\Omega)},
\end{aligned}
\end{equation*}
where the Poincar\'{e} inequality is used in the last inequality.
Therefore,
\begin{equation}\label{nablaPhUh}
\begin{aligned}
\|\nabla(P_hu-u_h)\|_{L^2(\Omega)} &\leq C\big(\lambda \|u-u_h\|_{L^2(\Omega)}+|\lambda-\lambda_h|\big)\\
 &\leq C\big(\lambda C h^{2} + Ch^{2}\big)|u|_{H^{2}(\Omega)}\\
 &\leq C h^{2}|u|_{H^{2}(\Omega)}
\end{aligned}
\end{equation}
using Lemma \ref{basicErrEst}.
\end{proof}

\begin{lemma}\label{Lemma24}
Assume that $u\in W^{2,4}(\Omega)$. Then,
\begin{equation}
\|\nabla u-\nabla u_h\|_{L^4(\Omega)} \leq C h|u|_{W^{2,4}(\Omega)}.
\end{equation}
\end{lemma}
\begin{proof}
By triangle inequality, we have
\begin{equation}\label{trianIn}
\|\nabla u- \nabla u_{h}\|_{L^{4}(\Omega)}\leq \|\nabla u- \nabla P_h u\|_{L^{4}(\Omega)}+\|\nabla P_h u- \nabla u_{h}\|_{L^{4}(\Omega)}.
\end{equation}
By (8.5.4) on pp. 230 \cite{Brenner} and the approximation property (4.4.28) on pp. 110 \cite{Brenner},
\begin{equation}\label{trianIn1}
\begin{aligned}
\|\nabla u- \nabla P_h u\|_{L^{4}(\Omega)}&\leq C\inf_{v\in V_h}\|\nabla u- \nabla v\|_{L^{4}(\Omega)}\\
&\leq Ch|u|_{W^{2,4}(\Omega)},
\end{aligned}
\end{equation}
By inverse inequality and Lemma \ref{basicErrEst1}, we have
\begin{equation}\label{trianIn2}
\begin{aligned}
\|\nabla P_h u- \nabla u_{h}\|_{L^{4}(\Omega)}&\leq Ch^{-\frac{d}{4}}\|\nabla P_h u- \nabla u_{h}\|_{L^{2}(\Omega)}\\
&\leq Ch^{2-\frac{d}{4}}|u|_{H^2(\Omega)}.
\end{aligned}
\end{equation}
A combination of (\ref{trianIn}), (\ref{trianIn1}) and (\ref{trianIn2}) allows the conclusion to hold.
\end{proof}

\subsection{Shape sensitivity analysis}
As a tool in shape optimization, shape calculus/shape sensitivity analysis can be performed by the \emph{velocity (speed) method} \cite{Defour,Sok1} and the \emph{perturbation of identity method} \cite{Pironneau}. The two approaches are equivalent in the sense of first-order expansion with respect to domain perturbations.
We recall basic shape calculus using the speed method (Sec. 2.9, pp. 54 and pp. 98 of \cite{Sok1}).

For a variable $t\in[0,\tau)$ with $\tau>0$, we introduce a velocity field $\mathcal{V}(t,x)\in C([0,\tau];$ $\mathcal{D}^1(\mathbb{R}^d,\mathbb{R}^d))$ with $\mathcal{D}^1(\mathbb{R}^d,\mathbb{R}^d)$ being the space of continuously differentiable transformations of $\mathbb{R}^d$. Then, we define a family of  transformations $T_t:\Omega\rightarrow \Omega_t$ with $\Omega_t=T_t(\mathcal{V})(\Omega)$. For $x=x(t,X)\in \Omega_t$ with $X\in\Omega$, it satisfies the following flow system
\begin{equation}
\frac{{\rm d}x}{{\rm d}t}(t,X) = \mathcal{V}(t,x(t,X)),\ \ x(0,X) = X.
\end{equation}
For some domain $\Omega$, a \emph{shape functional} depending on the shape is denoted by $J(\Omega)$ with $J(\cdot):\Omega\mapsto\mathbb{R}$. Denote $\mathcal{V}=\mathcal{V}(0,X)$ in the following for simplicity.
\begin{definition}
The \emph{Eulerian derivative} of $J(\Omega)$ at $\Omega$ in the direction $\mathcal{V}$ is defined by
\begin{equation}
{\rm d}J(\Omega;\mathcal{V}):=\lim_{t\searrow 0} \frac{J(\Omega_t)-J(\Omega)}{t}
\end{equation}
if the limit exists \cite{Defour}.
\end{definition}
\begin{definition}
The shape functional $J(\Omega)$ is called \emph{shape differentiable} at $\Omega$ if
(i) there exist Eulerian derivatives for all directions $\mathcal{V}$;\\
(ii) the map $\mathcal{V}\rightarrow {\rm d}J(\Omega;\mathcal{V})$ is linear and continuous from $C([0,\tau];\mathcal{D}^1(\mathbb{R}^d,\mathbb{R}^d))$ to $\mathbb{R}$.
\end{definition}

We remark that non-differentiable cases occur when, e.g., the Eulerian derivative exists but the mapping $\mathcal{V}\mapsto {\rm d}J(\Omega;\mathcal{V})$ is nonlinear. Such cases occur for the shape functionals of multiple eigenvalues.
\begin{definition}
The material derivative in some Sobolev space ${W}(\Omega)$ of a state variable $u=u(\Omega)\in {W}(\Omega)$ in a direction $\mathcal{V}$ is denoted as
\begin{equation}
\dot{u}(\Omega;\mathcal{V}):=\lim_{t\searrow 0} \frac{u(\Omega_t)\circ T_t(\mathcal{V})(\Omega)-u(\Omega)}{t}
\end{equation}
if the limit exists.
\end{definition}

Material derivatives on the boundary $\partial\Omega$ can be defined analogously.
When taking into account the strong (or weak) convergence in ${W}(\Omega)$ for the limit, the material derivative will be more specified as a strong (or weak) version.
%

The structure theorem [10, Corollary 1, pp. 480] states that the boundary Eulerian derivative of shape functional depends only on the normal part of the velocity on the boundary when certain smoothness of the boundary is satisfied. The volume formulation of Eulerian derivative actually holds with less smoothness requirement on boundary \cite{Laurain} and offers more accuracy \cite{Hiptmair}.
By the \emph{speed method}, we can obtain the Eulerian derivatives of simple as well as multiple eigenvalues for both Dirichlet and Neumann boundary conditions. Denote $\mathcal{V}_n=\mathcal{V}(0)|_{\partial\Omega}\cdot n$. The Eulerian derivative of an eigenvalue $\lambda=\lambda(\Omega)$ (depending on $\Omega$) in the direction $\mathcal{V}$ is defined to be
\begin{equation}
{\rm d}\lambda(\Omega;\mathcal{V}):=\lim_{t\searrow 0} \frac{\lambda(\Omega_t)-\lambda(\Omega)}{t}.
\end{equation}
For a simple eigenvalue, let $(\lambda,u)$ be an eigenpair of the problem (\ref{varForm}). Then, $\lambda(\Omega)$ is shape differentiable and
\begin{equation}\label{ThmDerEq}
{\rm d}\lambda(\Omega;\mathcal{V}) = \int_\Omega \big[ - 2\nabla u\cdot {\rm D}\mathcal{V}\nabla u + {\rm div} \mathcal{V}(|\nabla u|^2-\lambda u^2)\big]{\rm d}x,
\end{equation}
where ${\rm D}\mathcal{V}$ denotes the Jacobian of $\mathcal{V}$.
If, furthermore, $\Omega$ is convex or if it is of class $C^2$, then the boundary Eulerian derivative of Dirichlet eigenvalue
\begin{equation}\label{ShapeDer}
{\rm d}\lambda(\Omega;\mathcal{V})=-\int_{\partial\Omega} \left(\frac{\partial u}{\partial n}\right)^2 \mathcal{V}_n{\rm d}s.
\end{equation}
If $\Omega$ is of class $C^3$ for the Neumann case, then the boundary Eulerian derivative
\begin{equation}\label{NeumannBdy}
{\rm d}\lambda(\Omega;\mathcal{V})=\int_{\partial\Omega} \left(|\nabla_{\Gamma} u|^2-\lambda u^2\right) \mathcal{V}_n{\rm d}s,
\end{equation}
where the \emph{tangential gradient}
\begin{equation*}
\nabla_\Gamma u := \nabla u - \frac{\partial u}{\partial n}n.
\end{equation*}
(See in Appendix the formal derivations of (\ref{ThmDerEq})-(\ref{NeumannBdy}) for self-containedness of the paper).

For $u$ sufficiently regular, the expression (\ref{ThmDerEq}) corresponds to those appearing concise in \cite{HenriBook}:
\begin{equation}\label{divExp}
\begin{aligned}
{\rm d}\lambda(\Omega;\mathcal{V}) =& -\int_\Omega {\rm div}\big( |\nabla u|^2 \mathcal{V}\big){\rm d}x \quad {\rm for \quad Dirichlet}\\
&\bigg({\rm or}\ \int_\Omega {\rm div}\big( (|\nabla u|^2-\lambda u^2) \mathcal{V}\big){\rm d}x \quad{\rm for\ \ Neumann}\bigg).
\end{aligned}
\end{equation}
However, it is not appropriate to use (\ref{divExp}) for discretization. The usual $C^0$ Lagrange finite element discretization of (\ref{divExp}) fails to hold since $\nabla u_h$ is not continuously differentiable. We thus consider (\ref{ThmDerEq}) for discretization.
%

For the multiple eigenvalue case, we simplify $\lambda_{k_i}$ as $\lambda$ and let $u_i$ $(i=1,2,\cdots,l)$ be its eigenfunctions satisfying (\ref{contNorm}). Then, $\lambda$ is no longer shape differentiable. Two strategies can be considered: the sub-differential and directional derivatives \cite{HenriBook,Rousselet83,Sok1}. We adopt the latter to follow closely the derivations of directional derivatives for the Dirichlet case as in Theorem 2.5.8 \cite{HenriBook} or \cite{Rousselet83}. Then we can have the following results for both boundary conditions:
\begin{proposition}\label{MultExp}
Assume that $\lambda=\lambda(\Omega)$ is a multiple eigenvalue of order $l\geq 2$ for (\ref{EigModel}) with Dirichlet boundary condition. Let $u_{1}, u_{2}\cdots,u_{l}$ be an $L^2$-orthonormal basis of the eigenspace associated with $\lambda$, then the Eulerian derivative is one of the eigenvalues of the matrix $\mathcal{M}\in \mathbb{R}^{l\times l}$ with the entry
\begin{equation}\label{MultiThmDerEq}
m_{i,j}  = \int_\Omega \big[ - ({\rm D}\mathcal{V}+{\rm D}\mathcal{V}^T)\nabla u_{i}\cdot \nabla u_{j} + {\rm div} \mathcal{V}(\nabla u_{i}\cdot \nabla u_{j}-\lambda  u_{i} u_{j})\big]{\rm d}x.
\end{equation}
If, furthermore, $\Omega$ is convex or if it is of class $C^2$, then
\begin{equation}\label{MultiShapeDer}
m_{i,j} =-\int_{\partial\Omega} \frac{\partial u_{i}}{\partial n}\frac{\partial u_{j}}{\partial n} \mathcal{V}_n{\rm d}s.
\end{equation}
In the case of Neumann boundary condition,
\begin{equation}\label{MultiNeumannVol}
m_{i,j}  = \int_\Omega \big[- ({\rm D}\mathcal{V}+{\rm D}\mathcal{V}^T)\nabla u_{i}\cdot \nabla u_{j} + {\rm div} \mathcal{V}(\nabla u_{i}\cdot \nabla u_{j}-\lambda  u_{i} u_{j})\big]{\rm d}x.
\end{equation}
If $\Omega$ is of class $C^3$, then
\begin{equation}\label{MultiNeumannBdy}
m_{i,j} =\int_{\partial\Omega} \left(\nabla_{\Gamma} u_{i}\cdot\nabla_{\Gamma} u_{j}-\lambda u_{i}u_{j}\right)\mathcal{V}_n{\rm d}s.
\end{equation}
with $i,j=1,\cdots,l$.
\end{proposition}
%

\section{\emph{A priori} error estimates of approximate shape gradients in Eulerian derivatives}
With the Galerkin finite element method for discretizations of the Laplace eigenvalue problem, we compute the approximate Eulerian derivatives and resulting shape gradients. We will analyze the convergence rates with {\emph a priori} error estimates in an infinite-dimensional operator norm. For simplicity, we will only discuss the Dirichlet case. The results below however can be similarly extended to the Neumann case. We will consider both cases of simple and multiple eigenvalues. We first discuss the case of simple eigenvalues. In last section, ${\rm d}\lambda(\Omega;\mathcal{V})$ is now simplified as $\lambda^\prime(\Omega;\mathcal{V})$. In order to differentiate notations for the boundary and volume type Eulearian derivatives, we denote (\ref{ThmDerEq}) and (\ref{ShapeDer}) by $\lambda^\prime(\Omega;\mathcal{V})_\Omega$ and $\lambda^\prime(\Omega;\mathcal{V})_{\partial\Omega}$, respectively. The finite element approximations of (\ref{ThmDerEq}) and (\ref{ShapeDer}) then read respectively:
\begin{equation}\label{discVV}
\lambda^\prime(\Omega;\mathcal{V})_{\Omega,h} := \int_\Omega \Big[ - 2 \nabla u_h\cdot {\rm D}\mathcal{V}\nabla u_h +  {\rm div} \mathcal{V}(|\nabla u_h|^2-\lambda_h u_h^2)\Big]{\rm d}x
\end{equation}
and
\begin{equation}\label{discBB}
\lambda^\prime(\Omega;\mathcal{V})_{\partial\Omega,h} := -\int_{\partial\Omega} \left(\frac{\partial u_h}{\partial n}\right)^2 \mathcal{V}_n{\rm d}s.
\end{equation}
In the continuous setting, $\lambda^\prime(\Omega;\mathcal{V})_{\Omega}=\lambda^\prime(\Omega;\mathcal{V})_{\partial\Omega}$ if $\partial\Omega$ is $C^2$. With $(\lambda,u)$ being discretized by finite elements, we have $\lambda^\prime(\Omega;\mathcal{V})_{\Omega,h}\neq\lambda^\prime(\Omega;\mathcal{V})_{\partial\Omega,h}$.

For the case of multiple eigenvalues, denote the matrices (resp. eigenvalues) $\mathcal{M}_{\Omega}$ (resp. $\{\sigma_{\Omega,i}\}_{i=1}^l$) and $\mathcal{M}_{\partial\Omega}$ (resp. $\{\sigma_{\partial\Omega,i}\}_{i=1}^l$) corresponding to (\ref{MultiThmDerEq}) and (\ref{MultiShapeDer}), respectively. The approximations of (\ref{MultiThmDerEq}) and (\ref{MultiShapeDer}) are
\begin{equation}\label{FEM_MultiThmDerEq}
m^h_{i,j}(\Omega;\mathcal{V})_\Omega  := \int_\Omega \Big[ - ({\rm D}\mathcal{V}+{\rm D}\mathcal{V}^T)\nabla u_{i,h}\cdot \nabla u_{j,h} + {\rm div} \mathcal{V}(\nabla u_{i,h}\cdot \nabla u_{j,h}-\lambda  u_{i,h} u_{j,h})\Big]{\rm d}x
\end{equation}
and
\begin{equation}\label{FEM_MultiShapeDer}
m^h_{i,j}(\Omega;\mathcal{V})_{\partial\Omega} :=-\int_{\partial\Omega} \frac{\partial u_{i,h}}{\partial n}\frac{\partial u_{j,h}}{\partial n} \mathcal{V}_n{\rm d}s
\end{equation}
with $i,j=1,2,\cdots,l$. The corresponding matrices (resp. eigenvalues) are denoted by $\mathcal{M}^h_{\Omega}$ (resp. $\{\sigma_{\Omega,i}^h\}_{i=1}^l$) and $\mathcal{M}^h_{\partial\Omega}$ (resp. $\{\sigma_{\partial\Omega,i}^h\}_{i=1}^l$), respectively.

For each simple/multiple eigenvalue case, \emph{a priori} error estimates are presented for two type (volume and boundary) finite element approximations of Eulerian derivatives and corresponding shape gradients. We first consider the case of simple eigenvalues and then the multiple case.

\begin{remark}
In most cases for shape gradient algorithms, the domain is polyhedral and thus no geometric errors are introduced after triangulations. We will not consider the geometric approximation errors when performing convergence analysis of approximate Eulerian derivatives. When the domain is smooth, e.g., $C^2$, for boundary formulas of Eulerian derivatives to hold, we assume that the geometric errors can be negligible by using isoparametric finite elements or fine meshes on boundaries.
\end{remark}

\subsection{Simple eigenvalue case}
For the continuous formulas (\ref{ThmDerEq})-(\ref{ShapeDer}), we present convergence analysis of the approximate Eulerian derivatives with the volume integral (\ref{discVV}) and boundary integral (\ref{discBB}), respectively. For the volume type, we have
\begin{theorem}\label{Thm:domain}
Let assumptions in Lemma \ref{basicErrEst} hold. Let $(\lambda,u)$ be a single eigenpair of (\ref{varForm}) and $(\lambda_h,u_h)$ be its Galerkin Lagrange finite element approximation in (\ref{FEMvarForm}). Then,
\begin{equation}\label{DomainErrEst}
|\lambda^\prime(\Omega;\mathcal{V})_\Omega - \lambda^\prime(\Omega;\mathcal{V})_{\Omega,h}| \leq Ch^{2s}|u|_{H^{1+s}(\Omega)}|\mathcal{V}|_{W^{1,\infty}(\Omega)}, \quad 0<s\leq 1.
\end{equation}
If $\mathcal{V}\in H^2(\Omega)^d$ and $u\in W^{2,4}(\Omega)$, we further have
\begin{equation}\label{DomainErrEst1}
|\lambda^\prime(\Omega;\mathcal{V})_\Omega - \lambda^\prime(\Omega;\mathcal{V})_{\Omega,h}| \leq Ch^{2}|u|_{W^{2,4}(\Omega)}|\mathcal{V}|_{H^{2}(\Omega)}.
\end{equation}
\end{theorem}
\begin{proof}
First, we have by (\ref{ThmDerEq}), (\ref{discVV}) and the triangle inequality
\begin{equation}\label{threeTerms}
\begin{aligned}
&|\lambda^\prime(\Omega;\mathcal{V}) - \lambda^\prime(\Omega;\mathcal{V})_{\Omega,h}| \leq \bigg|\int_\Omega \big(2 \nabla u\cdot {\rm D}\mathcal{V}\nabla u- 2 \nabla u_h\cdot {\rm D}\mathcal{V}\nabla u_h\big) {\rm d}x\bigg| \\
&\quad \quad  \quad \quad+ \bigg|\int_\Omega {\rm div}  \mathcal{V}\big(|\nabla u|^2-|\nabla u_h|^2\big){\rm d}x\bigg|+ \bigg|\int_\Omega {\rm div}  \mathcal{V}\big(\lambda^2u^2-\lambda_h^2u_h^2\big){\rm d}x\bigg|.
\end{aligned}
\end{equation}
For the first term in R.H.S. of (\ref{threeTerms}),
\begin{equation}\label{threeFEw}
\begin{aligned}
&\ \ \bigg|\int_\Omega \Big(2 \nabla u\cdot {\rm D}\mathcal{V}\nabla u- 2 \nabla u_h\cdot {\rm D}\mathcal{V}\nabla u_h \Big){\rm d}x\bigg| \\
&=\bigg|\int_\Omega 2 (\nabla u- \nabla u_h) \cdot ({\rm D}\mathcal{V}+{\rm D}\mathcal{V}^T)\nabla u {\rm d}x -\int_\Omega 2 (\nabla u-\nabla u_h)\cdot {\rm D}\mathcal{V}(\nabla u-\nabla u_h) {\rm d}x\bigg|\\
&\leq\bigg|\int_\Omega 2 (\nabla u-\nabla u_h) \cdot ({\rm D}\mathcal{V}+{\rm D}\mathcal{V}^T)\nabla u{\rm d}x\bigg| +\bigg|\int_\Omega 2 (\nabla u-\nabla u_h)\cdot {\rm D}\mathcal{V}(\nabla u-\nabla u_h) {\rm d}x\bigg|\\
&\leq 2\|\nabla u- \nabla u_h\|_{H^{-s}(\Omega)}\|({\rm D}\mathcal{V}+{\rm D}\mathcal{V}^T)\nabla u\|_{H^{s}(\Omega)}+2\|{\rm D}\mathcal{V}\|_{L^\infty(\Omega)}\|\nabla u-\nabla u_h\|^2_{L^{2}(\Omega)}\\
&\leq C h^{2s}|u|^2_{H^{1+s}(\Omega)}|\mathcal{V}|_{W^{1,\infty}(\Omega)}
+C h^{2s}|u|^2_{H^{1+s}(\Omega)}|\mathcal{V}|_{W^{1,\infty}(\Omega)},
\end{aligned}
\end{equation}
where Lemma \ref{basicErrEst} is used in the last inequality.
For the second term in R.H.S. of (\ref{threeTerms}), we have analogously
\begin{equation}\label{SecErrEst}
\begin{aligned}
&\quad\bigg|\int_\Omega {\rm div} \mathcal{V}\big(|\nabla u|^2-|\nabla u_h|^2\big){\rm d}x\bigg|\\
&=\bigg|2\int_\Omega {\rm div} \mathcal{V} \nabla u\cdot(\nabla u-\nabla u_h){\rm d}x+\int_\Omega {\rm div} \mathcal{V} |\nabla u-\nabla u_h|^2{\rm d}x\bigg|\\
&\leq\bigg|2\int_\Omega {\rm div} \mathcal{V} \nabla u\cdot(\nabla u-\nabla u_h){\rm d}x\bigg|+\bigg|\int_\Omega {\rm div} \mathcal{V} |\nabla u-\nabla u_h|^2{\rm d}x\bigg|\\
&\leq 2\|{\rm div} \mathcal{V}\|_{L^\infty(\Omega)}\| \nabla u\|_{H^s(\Omega)}\| \nabla u- \nabla u_h\|_{H^{-s}(\Omega)}
+\|{\rm div}\mathcal{V}\|_{L^{\infty}(\Omega)} \|\nabla u-\nabla u_h\|_{L^2(\Omega)}^2\\
&\leq C h^{2s}|u|^2_{H^{1+s}(\Omega)}|\mathcal{V}|_{W^{1,\infty}(\Omega)}
+C h^{2s}|u|^2_{H^{1+s}(\Omega)}|\mathcal{V}|_{W^{1,\infty}(\Omega)}.
\end{aligned}
\end{equation}
For the third term in R.H.S. of (\ref{threeTerms}), simple estimations yield
\begin{equation}\label{fxfs}
\begin{aligned}
&\quad\bigg|\int_\Omega{\rm div} \mathcal{V}(\lambda^2u^2-\lambda_h^2u_h^2){\rm d}x\bigg|\\
&=\bigg| \int_\Omega {\rm div} \mathcal{V}[2\lambda u(\lambda u-\lambda_hu_h)-(\lambda u-\lambda_hu_h)^2]{\rm d}x\bigg| \\
&\leq \|{\rm div}\mathcal{V}\|_{L^{\infty}(\Omega)} \bigg(\int_\Omega |2\lambda u(\lambda u-\lambda_hu_h)|{\rm d}x+\int_\Omega|\lambda u-\lambda_hu_h|^2 {\rm d}x\bigg) \\
&\leq C |\mathcal{V}|_{W^{1,\infty}(\Omega)} \Big(2\lambda \|u\|_{L^2(\Omega)}\|\lambda u-\lambda_h u_h\|_{L^2(\Omega)}+\|\lambda u-\lambda_h u_h\|^2_{L^2(\Omega)}\Big),
\end{aligned}
\end{equation}
where Cauchy-Schwarz inequality is used.
In (\ref{fxfs}), $\|u\|_{L^2(\Omega)}=1$ and
\begin{equation}\label{fwef}
\begin{aligned}
\|\lambda u-\lambda_h u_h\|_{L^2(\Omega)}&=\|\lambda(u-u_h)+(\lambda-\lambda_h)u_h\|_{L^2(\Omega)}\\
&\leq \lambda\| u-u_h\|_{L^2(\Omega)} + |\lambda-\lambda_h|\|u_h\|_{L^2(\Omega)} \\
&\leq C\lambda h^{1+s}|u|_{H^{1+s}(\Omega)} + C h^{2s}|u|_{H^{1+s}(\Omega)}  
\end{aligned}
\end{equation}
with $\|u_h\|_{L^2(\Omega)}=1$ and Lemma \ref{basicErrEst} being used.
Therefore,
\begin{equation}\label{ThirdErrEst}
\begin{aligned}
\bigg|\int_\Omega{\rm div} \mathcal{V}(\lambda^2u^2-\lambda_h^2u_h^2){\rm d}x\bigg| &\leq C|\mathcal{V}|_{W^{1,\infty}(\Omega)} \Big(2\lambda\big( C \lambda h^{1+s} + C h^{2s}\big)|u|_{H^{1+s}(\Omega)}\\
&\quad + \big( C \lambda h^{1+s} + C h^{2s}\big)^2|u|_{H^{1+s}(\Omega)}^2 \Big)\\
&\leq C \lambda^2 h^{2s}|\mathcal{V}|_{W^{1,\infty}(\Omega)}|u|_{H^{1+s}(\Omega)}.
\end{aligned}
\end{equation}
Substituting (\ref{threeFEw}), (\ref{SecErrEst}) and (\ref{ThirdErrEst}) into (\ref{threeTerms}) allows (\ref{DomainErrEst}) to hold.

Now we prove (\ref{DomainErrEst1}) when $\mathcal{V}\in H^2(\Omega)^d$ and $u\in W^{2,4}(\Omega)$. Each term on R.H.S. of (\ref{threeTerms}) is estimated differently from above due to different regularities on $u$ and $\mathcal{V}$. By H\"{o}lder's inequality, Lemma \ref{basicErrEst}, Lemma \ref{Lemma24}, and Sobolev embedding theorem, we have
\begin{equation}\label{threeFEwof}
\begin{aligned}
&\ \ \bigg|\int_\Omega \Big(2 \nabla u\cdot {\rm D}\mathcal{V}\nabla u- 2 \nabla u_h\cdot {\rm D}\mathcal{V}\nabla u_h \Big){\rm d}x\bigg| \\
&\leq 2\|\nabla u- \nabla u_h\|_{H^{-1}(\Omega)}\|({\rm D}\mathcal{V}+{\rm D}\mathcal{V}^T)\nabla u\|_{H^{1}(\Omega)}\\
&\quad+2\|{\rm D}\mathcal{V}\|_{L^4(\Omega)}\|\nabla u-\nabla u_h\|_{L^{4}(\Omega)}\|\nabla u-\nabla u_h\|_{L^{2}(\Omega)}\\
&\leq C h^{2}|u|_{H^{2}(\Omega)}(|\mathcal{V}|_{H^{2}(\Omega)}
|u|_{W^{1,\infty}(\Omega)} + |\mathcal{V}|_{W^{1,4}(\Omega)}|u|_{W^{2,4}(\Omega)})
\\
&\quad+Ch^{2}|u|_{H^{2}(\Omega)}|u|_{W^{2,4}(\Omega)}|\mathcal{V}|_{W^{1,4}(\Omega)}\\
&\leq C h^{2}|u|_{H^{2}(\Omega)}|\mathcal{V}|_{H^{2}(\Omega)}|u|_{W^{2,4}(\Omega)}
+C h^{2}|u|_{H^{2}(\Omega)}|u|_{W^{2,4}(\Omega)}|\mathcal{V}|_{H^{2}(\Omega)}
\end{aligned}
\end{equation}
and
\begin{equation}
\begin{aligned}
&\quad\bigg|\int_\Omega {\rm div} \mathcal{V}\big(|\nabla u|^2-|\nabla u_h|^2\big){\rm d}x\bigg|\\
&\leq\bigg|2\int_\Omega {\rm div} \mathcal{V} \nabla u\cdot(\nabla u-\nabla u_h){\rm d}x\bigg|+\bigg|\int_\Omega {\rm div} \mathcal{V} |\nabla u-\nabla u_h|^2{\rm d}x\bigg|\\
&\leq 2\|{\rm div} \mathcal{V} \nabla u\|_{H^1(\Omega)}\| \nabla u- \nabla u_h\|_{H^{-1}(\Omega)}\\
&\quad+\|{\rm div}\mathcal{V}\|_{L^{4}(\Omega)} \|\nabla u-\nabla u_h\|_{L^4(\Omega)}\|\nabla u-\nabla u_h\|_{L^2(\Omega)}\\
&\leq C h^{2}|u|_{H^{2}(\Omega)}\|{\rm div} \mathcal{V} \nabla u\|_{H^1(\Omega)}
+C h^{2}|u|_{H^{2}(\Omega)}|u|_{W^{2,4}(\Omega)}|\mathcal{V}|_{W^{1,4}(\Omega)}\\
&\leq C h^{2}|u|_{H^{2}(\Omega)}(|\mathcal{V}|_{W^{1,4}(\Omega)}| u|_{W^{2,4}(\Omega)}+|\mathcal{V}|_{H^{2}(\Omega)}|u|_{W^{1,\infty}(\Omega)}
+|u|_{W^{2,4}(\Omega)}|\mathcal{V}|_{W^{1,4}(\Omega)})\\
&\leq C h^{2}|u|_{H^{2}(\Omega)}|\mathcal{V}|_{H^{2}(\Omega)}|u|_{W^{2,4}(\Omega)},
\end{aligned}
\end{equation}
respectively.

For the third term in R.H.S. of (\ref{threeTerms}), simple estimations by H\"{o}lder's inequality, Sobolev embedding theorem, and (\ref{fwef}) with $s=1$ yield
\begin{equation}
\begin{aligned}
&\quad\bigg|\int_\Omega{\rm div} \mathcal{V}(\lambda^2u^2-\lambda_h^2u_h^2){\rm d}x\bigg|\\
&\leq \|{\rm div}\mathcal{V}\|_{L^{4}(\Omega)} \|\lambda u-\lambda_hu_h\|_{L^2(\Omega)}\|\lambda u+\lambda_hu_h\|_{L^4(\Omega)} \\
&\leq C |\mathcal{V}|_{H^{2}(\Omega)} h^2|u|_{H^2(\Omega)}\|\lambda u+\lambda_hu_h\|_{H^1(\Omega)},
\end{aligned}
\end{equation}
in which
\begin{equation*}
\begin{aligned}
&\quad\|\lambda u+\lambda_hu_h\|_{H^1(\Omega)}\\
&\leq \lambda\|u\|_{H^1(\Omega)}+\lambda_h\|u_h\|_{H^1(\Omega)} \\
&= \lambda\sqrt{\lambda} + \lambda_h\sqrt{\lambda_h}\\
&\leq \lambda\sqrt{\lambda} + (\lambda+Ch^2|u|_{H^2(\Omega)})^{\frac{3}{2}}.
\end{aligned}
\end{equation*}
Thus, the conclusion follows by unifying the above results.
\end{proof}
\begin{remark}
In \cite{Hiptmair}, $H^2$ regularity is assumed for convergence analysis of approximate shape gradients in linear elliptic problems. We remark that the $H^2$ regularity fails to hold since the polyhedral domain, e.g., polygon, may easily lose convexity during shape evolutions. Therefore, it is reasonable to assume less regularity in convergence analysis of approximate shape gradients. In Theorem \ref{Thm:domain}, the more general regularity $H^{1+s}$ ($0<s\leq 1$) assumption is thus made for (\ref{DomainErrEst}). Moreover, another new result (\ref{DomainErrEst1}) is obtained under more regularity $W^{2,4}$ on $u$ and different $H^2$ regularity instead of $W^{1,\infty}$ on $\mathcal{V}$. This kind of result is absent in \cite{Hiptmair} and will be used for convergence analysis in Section 4 below.
\end{remark}

Now, we perform convergence analysis for the approximate boundary Eulerian derivative (\ref{discBB}). We have to first assume that $\Omega$ is convex $(d=2)$ or $C^2$ such that $u\in H^2(\Omega)$ for the continuous boundary formula (\ref{MultiShapeDer}) to hold.
\begin{theorem}\label{Thm:Boundary}
Let the assumptions in Theorem \ref{Thm:domain} hold with $s=1$. Assume further that
\begin{equation*}
\|u\|_{W^{2,p}(\Omega)}\leq C p \lambda\|u\|_{L^p(\Omega)}
\end{equation*}
for $1<p<\mu$ with some $\mu>d$.
Then,
\begin{equation*}\label{errfa}
|\lambda^\prime(\Omega;\mathcal{V})_{\partial\Omega} - \lambda^\prime(\Omega;\mathcal{V})_{\partial\Omega,h}| \leq
C|\log{h}|^{1-\frac{1}{d}}h|u|_{W^{2,\infty}(\Omega)}\|\mathcal{V}_n\|_{L^{\infty}(\partial\Omega)}.
\end{equation*}
\end{theorem}
\begin{proof}
By (\ref{ShapeDer}) and (\ref{discBB}), we first have
\begin{equation}\label{efwfoj}
\begin{aligned}
&\quad |\lambda^\prime(\Omega;\mathcal{V}) - \lambda^\prime(\Omega;\mathcal{V})_{\partial\Omega,h}|\\
&=\left|\int_{\partial\Omega}-\left[\left(\frac{\partial u}{\partial n}\right)^2-\left(\frac{\partial u_h}{\partial n}\right)^2\right]\mathcal{V}_n {\rm d}s\right|\\
&\leq \|\mathcal{V}_n\|_{L^\infty(\partial\Omega)}\int_{\partial\Omega}\left|\left(\frac{\partial u}{\partial n}\right)^2-\left(\frac{\partial u_h}{\partial n}\right)^2\right|{\rm d}s \\
&= \|\mathcal{V}_n\|_{L^\infty(\partial\Omega)}\int_{\partial\Omega}\left|2\frac{\partial u}{\partial n}\frac{\partial (u-u_h)}{\partial n}-\left[\frac{\partial (u-u_h)}{\partial n}\right]^2\right|{\rm d}s\\
&\leq \|\mathcal{V}_n\|_{L^\infty(\partial\Omega)}\left(2\left\|\frac{\partial u}{\partial n}\right\|_{L^1(\partial\Omega)}\left\|\frac{\partial (u-u_h)}{\partial n}\right\|_{L^\infty(\partial\Omega)}+|\partial\Omega|\left\|\frac{\partial (u-u_h)}{\partial n}\right\|_{L^\infty(\partial\Omega)}^2\right)
\end{aligned}
\end{equation}
where we have used the H\"{o}lder inequality in the last inequality. By the trace theorem \cite{Adams}, (\ref{efwfoj}) implies that
\begin{equation}\label{joiewfw}
\begin{aligned}
&\quad |\lambda^\prime(\Omega;\mathcal{V}) - \lambda^\prime(\Omega;\mathcal{V})_{\partial\Omega,h}|\\
&\leq C\|\mathcal{V}_n\|_{L^\infty(\partial\Omega)}\bigg(2\left\|\frac{\partial u}{\partial n}\right\|_{L^1(\partial\Omega)}\|u-u_h\|_{W^{1,\infty}(\Omega)}
+|\partial\Omega|\|u-u_h\|_{W^{1,\infty}(\Omega)}^2\bigg).
\end{aligned}
\end{equation}
Then, the conclusion follows using the \emph{a priori} error estimates in the norm $W^{1,\infty}(\Omega)$ \cite{Brenner}:
\begin{equation}\label{W1inftyErrEst}
\|u-u_h\|_{W^{1,\infty}(\Omega)}=
C|\log{h}|^{1-\frac{1}{d}}h|u|_{W^{2,\infty}(\Omega)}.
\end{equation}
What left now is to prove (\ref{W1inftyErrEst}).
First, we split the error and use the triangle inequality to obtain
\begin{equation}\label{xxx}
\|u-u_h\|_{W^{1,\infty}(\Omega)}\leq \|u-P_hu\|_{W^{1,\infty}(\Omega)}+\|P_hu-u_h\|_{W^{1,\infty}(\Omega)},
\end{equation}
where $P_h: H^1_0(\Omega)\rightarrow V_h$ is defined in (\ref{RitzDef}).
In (\ref{xxx}), the error estimate for the first term on the R.H.S. is standard (Corollary 8.1.12 \cite{Brenner}):
\begin{equation}\label{standard}
\|u-P_h u\|_{W^{1,\infty}(\Omega)}\leq
C h|u|_{W^{2,\infty}(\Omega)}.
\end{equation}
To estimate $\|P_hu-u_h\|_{W^{1,\infty}(\Omega)}$, the methods we use are different for $d=2$ and $d=3$. We discuss them separately. For $d=2$,
by the inverse inequality (see e.g., \cite{Brenner}), discrete Sobolev inequality (Lemma 4.9.2 of \cite{Brenner}) and Lemma \ref{basicErrEst1}, we obtain
\begin{equation}\label{ffx}
\begin{aligned}
\|P_hu-u_h\|_{W^{1,\infty}(\Omega)}&\leq Ch^{-1}\|P_hu-u_h\|_{L^{\infty}(\Omega)}\\
&\leq Ch^{-1}|\log{h}|^{1/2}\|\nabla (P_hu-u_h)\|_{L^{2}(\Omega)}\\
&\leq C |\log{h}|^{1/2}h|u|_{H^{2}(\Omega)},\quad d=2.
\end{aligned}
\end{equation}

For $d=3$, we have to use different arguments since no result as Lemma 4.9.2 of \cite{Brenner} is available. Let $p<3$.
By the inverse inequality \cite{Brenner}, Theorem 7.10 and its remark in \cite{2ndPDE}
\begin{equation}\label{ffa}
\begin{aligned}
\|P_hu-u_h\|_{W^{1,\infty}(\Omega)} &\leq Ch^{-1-\frac{3-p}{p}}\|P_hu-u_h\|_{L^{{3p}/{(3-p)}}(\Omega)}\\
&\leq CC_{p}h^{-1-\frac{3-p}{p}}\|\nabla(P_hu-u_h)\|_{L^p(\Omega)},
\end{aligned}
\end{equation}
where
\begin{equation*}
C_{p}=\frac{1}{3\sqrt{\pi}}\bigg(\frac{3 !\Gamma{(3/2)}}{2\Gamma{(3/p)}\Gamma{(4-3/p)}}\bigg)^{1/3}\bigg[\frac{3(p-1)}{3-p}\bigg]^{1-1/p}
\end{equation*}
with $\Gamma(\cdot)$ denoting the gamma function.
Using H\"{o}lder's inequality
\begin{equation}\label{Holder}
\|v\|_{L^p(\Omega)} \leq |\Omega|^{\frac{3-p}{3p}}\|v\|_{L^3(\Omega)}\quad \forall v\in L^3(\Omega)
\end{equation}
in (\ref{ffa}) and choosing $p$ such that $(3-p)|\log{h}|=p$, we get
\begin{equation}\label{dim3ErrEst}
\begin{aligned}
\|P_hu-u_h\|_{W^{1,\infty}(\Omega)} &\leq Ch^{-1-\frac{3-p}{p}}|\Omega|^{\frac{3-p}{3p}}\|\nabla(P_hu-u_h)\|_{L^3(\Omega)}\\
& \leq C h^{-1}|\log{h}|^{2/3} \|\nabla(P_hu-u_h)\|_{L^3(\Omega)}.
\end{aligned}
\end{equation}
Since
\begin{equation*}
\begin{aligned}
\|\nabla(P_hu-u_h)\|_{L^3(\Omega)}&\leq C \sup_{0\neq v_h\in V_h} \frac{(\nabla(P_hu-u_h),\nabla v_h)}{|v_h|_{W^{1,3/2}(\Omega)}} \quad ({\rm Proposition}\ 8.6.2 \ [12])\\
&=C\sup_{0\neq v_h\in V_h} \frac{\big(\lambda(u-u_h)+(\lambda-\lambda_h)u_h,v_h\big)}{|v_h|_{W^{1,3/2}(\Omega)}}\\
&\leq C\sup_{0\neq v_h\in V_h} \frac{\big(\lambda \|u-u_h\|_{L^2(\Omega)}+|\lambda-\lambda_h|\big)\|v_h\|_{L^2(\Omega)}}{|v_h|_{W^{1,3/2}(\Omega)}}\\
&\leq Ch^{2}|u|_{H^{2}(\Omega)}\sup_{0\neq v_h\in V_h} \frac{\|v_h\|_{L^2(\Omega)}}{|v_h|_{W^{1,3/2}(\Omega)}}  \quad ({\rm Lemma}\ \ref{basicErrEst}) \\
&\leq Ch^{2}|u|_{H^{2}(\Omega)}\sup_{0\neq v_h\in V_h} \frac{\|v_h\|_{L^2(\Omega)}}{\|v_h\|_{L^3(\Omega)}}  \\
&\leq  Ch^{2}|u|_{H^{2}(\Omega)}|\Omega|^{1/6}
\end{aligned}
\end{equation*}
using Sobolev embedding Theorem and (\ref{Holder}) in the last two inequalities,
(\ref{dim3ErrEst}) implies that
\begin{equation}\label{efwio}
\|P_hu-u_h\|_{W^{1,\infty}(\Omega)}= C|\Omega|^{1/6}|\log{h}|^{2/3}h|u|_{H^{2}(\Omega)},\quad d=3.
\end{equation}
Finally, a combination of (\ref{standard}), (\ref{ffx}) and (\ref{efwio}) allows us to arrive at (\ref{W1inftyErrEst}).
\end{proof}

\begin{remark}
  Comparing with Theorem $\ref{Thm:domain}$, more regularity on $u$ is required in Theorem \ref{Thm:Boundary}.
  However, the converge rate (interpreted as $\mathcal{O}(h^{1-\epsilon})$ for any small $\epsilon>0$) obtained in Theorem \ref{Thm:Boundary} is lower than $\mathcal{O}(h^{2})$ in Theorem $\ref{Thm:domain}$.
\end{remark}

\subsection{Multiple eigenvalue case}
Now we turn to the more complicated multiple eigenvalue case. We show the \emph{a priori} error estimates for the approximate volume and boundary Eulerian derivatives. The directional derivatives for this case are stated in Proposition \ref{MultExp}. We first recall Weyl's inequality for estimating perturbations of the spectrum in matrix theory \cite{Zhan}.
\begin{lemma}\label{MatrixPerturbation}
Let matrices $\mathcal{A}=[a_{ij}]$ and $\mathcal{A}^h=[a_{ij}^h]\in \mathbb{R}^{l\times l}$ be symmetric. If the entries satisfy that $|a_{ij}-a_{ij}^h|=\mathcal{O}(h^\vartheta)$ with some $\vartheta>0$ for $i,j=1,2,\cdots,l$. Denote by $\{\theta_i\}_{i=1}^l$ (resp. $\{\theta_i^h\}_{i=1}^l$) the eigenvalues of $\mathcal{A}$ (resp. $\mathcal{A}^h$). Then, 
\begin{equation}
\max_{1\leq i\leq l} |\theta_i-\theta_i^h| =\mathcal{O} (l^{3/2}h^\vartheta).
\end{equation}
\end{lemma}
\begin{proof}
We first obtain symmetry of $\mathcal{A}-\mathcal{A}^h$ since both $\mathcal{A}$ and $\mathcal{A}^h$ are symmetric. Then, the spectral norm and eigenvalues satisfy that $\|\mathcal{A}-\mathcal{A}^h\|_2=\max_{i}|\theta_i-\theta_i^h|$. For matrix norms, we easily have \cite{Demmel}
\begin{equation*}
\|\mathcal{A}-\mathcal{A}^h\|_2\leq \sqrt{l}\|\mathcal{A}-\mathcal{A}^h\|_\infty.
\end{equation*}
Thus,
\begin{equation*}
\begin{aligned}
\max_{1\leq i\leq l}|\theta_i-\theta_i^h| & \leq \sqrt{l}\|\mathcal{A}-\mathcal{A}^h\|_\infty  \\
& = \sqrt{l}\max_{1\leq i\leq l}\sum_{j=1}^l |a_{ij}-a^h_{ij}|.
\end{aligned}
\end{equation*}
The result follows from the known condition on perturbation bounds of entries.
\end{proof}
Then, we obtain for approximations (\ref{FEM_MultiThmDerEq}) and (\ref{FEM_MultiShapeDer}) of multiple eigenvalues.
\begin{theorem}\label{MultiThm:Boundary}
Let the assumptions in Lemma \ref{basicErrEst} and Proposition \ref{MultExp} hold. Denote by $\{u_{i,h}\}_{i=1}^l$ the Lagrange finite element approximations of eigenfunctions $\{u_{i}\}_{i=1}^l$. Then,
\begin{equation*}
\max_{1\leq i \leq l}|\sigma_{\Omega,i} - \sigma_{\Omega,i}^h| \leq C l^{3/2} h^{2s}\max_{1\leq i\leq l}|u_i|_{H^{1+s}(\Omega)}|\mathcal{V}|_{W^{1,\infty}(\Omega)},\quad 0< s\leq 1.
\end{equation*}
If $\mathcal{V}\in H^2(\Omega)^d$ and moreover $u\in W^{2,4}(\Omega)$, then
\begin{equation*}
\max_{1\leq i \leq l}|\sigma_{\Omega,i} - \sigma_{\Omega,i}^h| \leq C l^{3/2} h^{2}\max_{1\leq i\leq l}|u_i|_{W^{2,4}(\Omega)}|\mathcal{V}|_{H^{2}(\Omega)}.
\end{equation*}
Assume that $u_i\in W^{2,\infty}(\Omega)$ for the boundary formula ($i=1,2,\cdots,l$). We have
\begin{equation*}
\max_{1\leq i\leq l}|\sigma_{\partial\Omega,i} - \sigma_{\partial\Omega,i}^h| \leq C l^{3/2} h^{1-\epsilon}\max_{1\leq i\leq l}|u_i|_{W^{2,\infty}(\Omega)}|\mathcal{V}_n|_{L^{\infty}(\partial\Omega)}, \quad 0< \epsilon\ll 1.
\end{equation*}
\end{theorem}
\begin{proof}
We can modify the arguments in Theorems \ref{Thm:domain} and \ref{Thm:Boundary} for the simple eigenvalue case to obtain that
\begin{equation*}
|m_{i,j}(\Omega;\mathcal{V})_{\Omega}-m^h_{i,j}(\Omega;\mathcal{V})_{\Omega}|\leq C h^{2s}
\max_{1\leq i\leq l}|u_i|_{H^{2}(\Omega)}|\mathcal{V}|_{W^{1,\infty}(\Omega)},
\end{equation*}
\begin{equation*}
|m_{i,j}(\Omega;\mathcal{V})_{\Omega}-m^h_{i,j}(\Omega;\mathcal{V})_{\Omega}| \leq C  h^{2}\max_{1\leq i\leq l}|u_i|_{W^{2,4}(\Omega)}|\mathcal{V}|_{H^{2}(\Omega)}
\end{equation*}
and
\begin{equation*}
|m_{i,j}(\Omega;\mathcal{V})_{\partial\Omega}-m^h_{i,j}(\Omega;\mathcal{V})_{\partial\Omega}|
\leq C h^{1-\epsilon}\max_{1\leq i\leq l}|u_i|_{W^{2,\infty}(\Omega)}|\mathcal{V}_n|_{L^{\infty}(\partial\Omega)}
\end{equation*}
for $i,j=1,\cdots,l$ and $\epsilon>0$. By Lemma \ref{MatrixPerturbation}, the conclusions follow.
\end{proof}
\begin{remark}
The results of Laplacian above may can be generalized for a self-adjoint and uniformly elliptic second-order differential operator $L$ such that
\begin{equation*}
Lu:=-\sum_{i,j=1}^d\frac{\partial}{\partial x_i}\left(a_{ij}(x)\frac{\partial u}{\partial x_j}\right) + a_0(x) u
\end{equation*}
with, e.g., coefficients $a_0(x)$ and $a_{ij}(x)\in C^1(\overline{\Omega})$ for $i,j = 1,\cdots,d$.
\end{remark}
\section{Numerical results} \label{Numer}
We perform numerical experiments with FreeFem++ \cite{Freefem}. We consider only the cases of simple eigenvalues for simplicity. Examples corresponding to both Dirichlet and Neumann boundary conditions are presented. We choose three computational domains in $\mathbb{R}^2$: unit square, unit disk and a L-shaped domain ($(-1,1)^2$ missing the upper right quarter). In Fig. \ref{FigMeshes}, one level of triangulation is illustrated. To study $h$-convergence, uniform refinement is employed. The eigenfunctions in the first two cases have enough smoothness and even can be extended to entire functions, whereas the eigenfunction on the L-shaped domain has a singularity at the reentrant corner. Lagrange Linear element is employed in all cases. We approximate the first eigenvalue and the first non-zero eigenvalue for the Dirichlet boundary condition and Neumann boundary condition, respectively.

We first verify numerically the theoretical results in Section 3. The shape gradient for simple eigenvalue is a linear and continuous operator on $W^{1,\infty}(\mathbb{R}^d;\mathbb{R}^d)$ and belongs to its dual space either in the volume or boundary type Eulerian derivative. As noted in \cite{Hiptmair}, it is challenging to compute numerically in the continuous infinite-dimensional operator norm for the approximate shape gradients. This norm can be approximately replaced by a more tractable one on a finite-dimensional subspace of $W^{1,\infty}(\mathbb{R}^d;\mathbb{R}^d)$.
More precisely, given a positive integer $\gamma$ as in \cite{Hiptmair}, we consider approximate operator norm on a finite-dimensional space consisting of vector fields in $\mathcal{P}_{\gamma,\gamma}(\mathbb{R}^d;\mathbb{R}^d)(\subset W^{1,\infty}(\mathbb{R}^d;\mathbb{R}^d))$, whose components are multivariate polynomials of degree not more than $\gamma$. We replace the $W^{1,\infty}$-norm with the more easily computable $H^1$-norm due to the equivalence of norms over finite-dimensional spaces. Finally, we compute the approximate dual norms
\begin{equation}\label{errorB}
\begin{aligned}
&\mathcal{E}_{\Omega}:=\left(
\max_{0\neq \mathcal{V}\in \mathcal{P}_{\gamma,\gamma}(\mathbb{R}^2;\mathbb{R}^2)}
\frac{|\lambda^\prime(\Omega;\mathcal{V})_{\Omega} - \lambda^\prime(\Omega;\mathcal{V})_{\Omega,h}|^2}{\|\mathcal{V}\|^2_{H^1(\Omega)}}\right)^{1/2},\\
&\mathcal{E}_{\partial\Omega}:=\left(
\max_{0\neq \mathcal{V}\in \mathcal{P}_{\gamma,\gamma}(\mathbb{R}^2;\mathbb{R}^2)}
\frac{|\lambda^\prime(\Omega;\mathcal{V})_{\partial\Omega} - \lambda^\prime(\Omega;\mathcal{V})_{\partial\Omega,h}|^2}{\|\mathcal{V}\|^2_{H^1(\Omega)}}\right)^{1/2}.
\end{aligned}
\end{equation}
We take a basis $\{\mathcal{V}_i\}_{i=1}^{q}$ of vector fields in $\mathcal{P}_{\gamma,\gamma}(\mathbb{R}^d;\mathbb{R}^d)$ with $q=dC_{\gamma+d}^d$ denoting the combination coefficient and
\begin{equation*}
\{\mathcal{V}_i\}_{i=1}^{q}=
\big\{[\Pi_{i=1}^d x_i^{\beta_i}, 0,\cdots,0],
\cdots,
[0, \cdots,0,\Pi_{i=1}^d x_i^{\beta_i} ]
\big\}_{\sum_{i=1}^d\beta_i\leq \gamma},
\end{equation*}
where $\beta_i$ ($i=1,\cdots,d$) are non-negative integers. Denote by $\mathbb{K}=[(\mathcal{V}_i,\mathcal{V}_j)_{H^1(\Omega)}]_{i,j=1}^{q}\in\mathbb{R}^{q\times q}$ be the Gramian matrix associated with the $H^1(\Omega)$ inner product.
For the simple eigenvalue case, the errors (\ref{errorB}) can be obtained by simply computing
\begin{equation}\label{wKw}
\mathcal{E} := \left(w^T\mathbb{K}^{-1}w\right)^{1/2},
\end{equation}
where $\mathcal{E}=\mathcal{E}_{\Omega}$ or $\mathcal{E}_{\partial\Omega}$ corresponds to $w=w_{\Omega}$ or $w_{\partial\Omega}$ with the vectors
\begin{equation*}
w_{\Omega}:=[\lambda^\prime(\Omega;\mathcal{V}_i)_{\Omega} - \lambda^\prime(\Omega;\mathcal{V}_i)_{\Omega,h}]_{i=1}^{q}\quad{\rm and}\quad w_{\partial\Omega}:=[\lambda^\prime(\Omega;\mathcal{V}_i)_{\partial\Omega} - \lambda^\prime(\Omega;\mathcal{V}_i)_{\partial\Omega,h}]_{i=1}^{q}.
\end{equation*}

\begin{figure}[htb]
\begin{minipage}[b]{0.33\textwidth}
\centering
\includegraphics[width=1.5in]{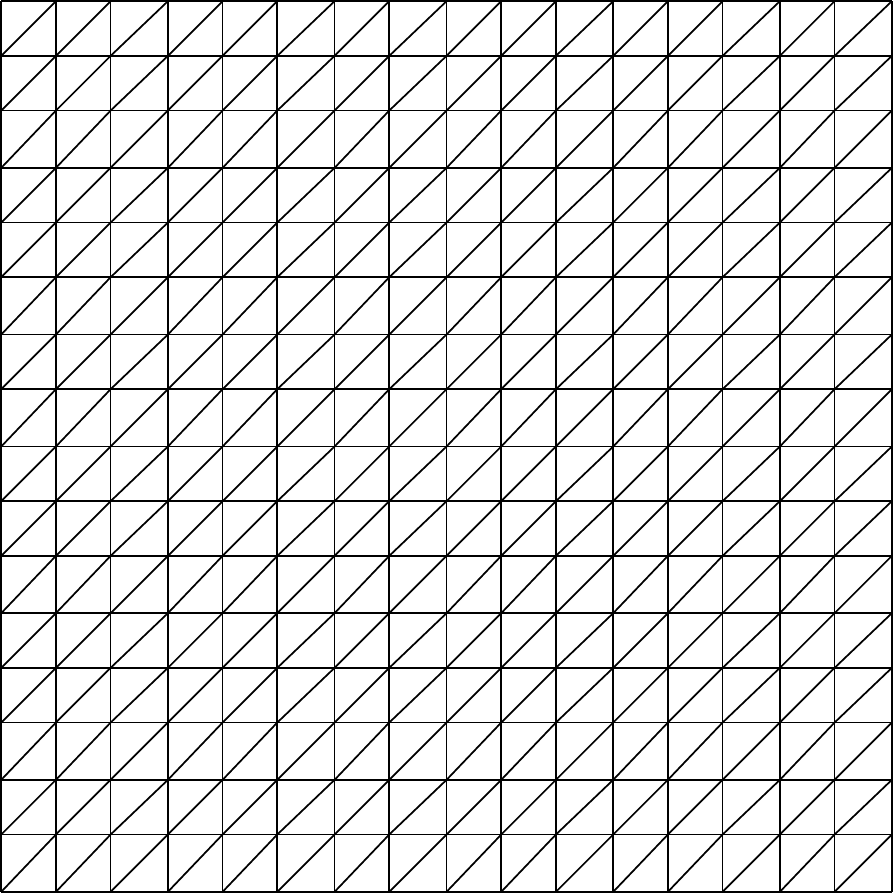}
\end{minipage}
\begin{minipage}[b]{0.33\textwidth}
\centering
\includegraphics[width=1.5in]{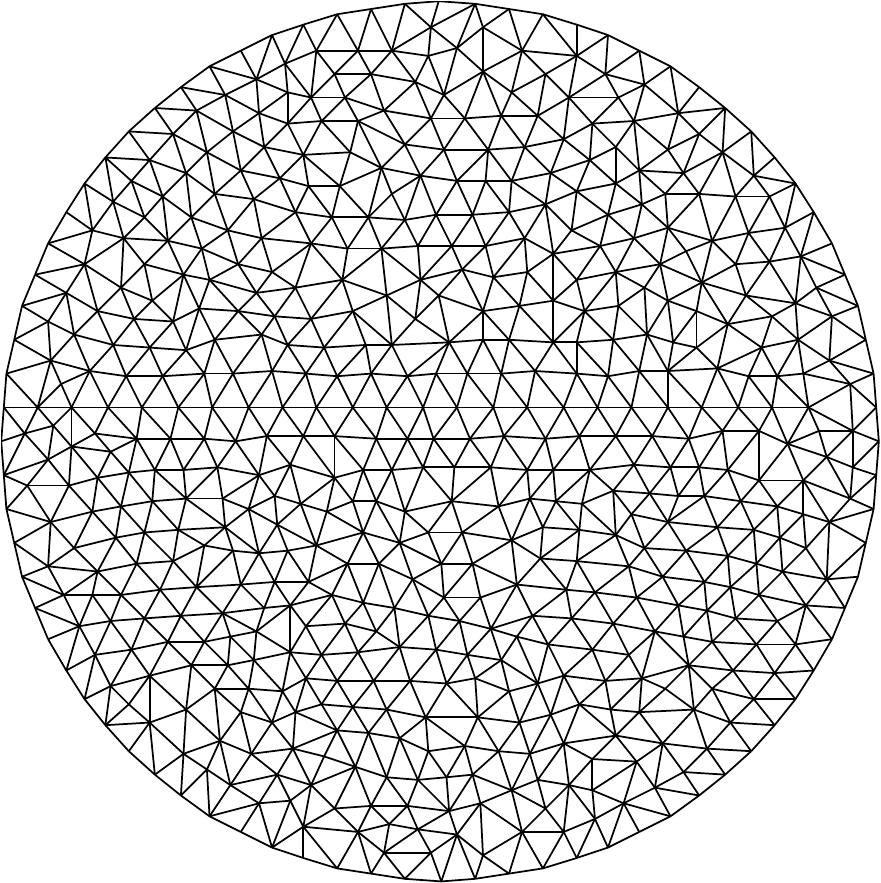}
\end{minipage}
\begin{minipage}[b]{0.33\textwidth}
\centering
\includegraphics[width=1.5in]{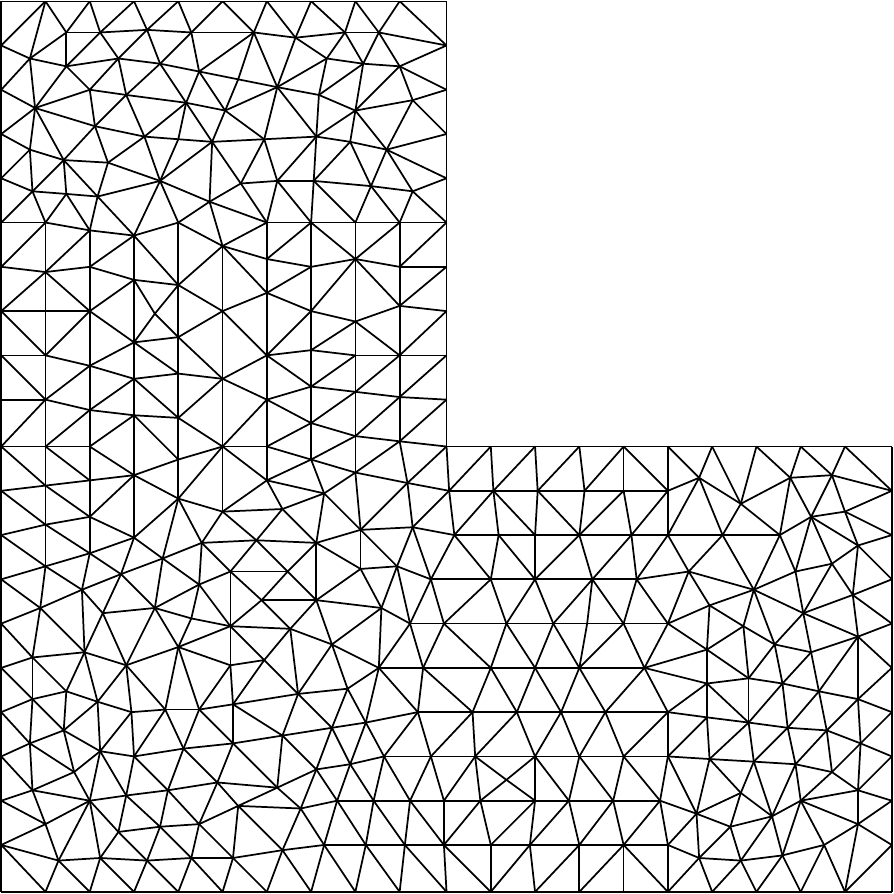}
\end{minipage}
\caption{One level of meshes used for the square, disk, and L-shaped domain.}
\label{FigMeshes} 
\end{figure}

\subsection{Dirichlet cases}
For square, uniform triangulation is used. The exact first eigenpair is
 $\big(2\pi^2,2\sin(\pi x_1)\sin(\pi x_2)\big)$.
\begin{figure}[htb]
\begin{minipage}[b]{0.49\textwidth}
\centering
\includegraphics[width=2.2in]{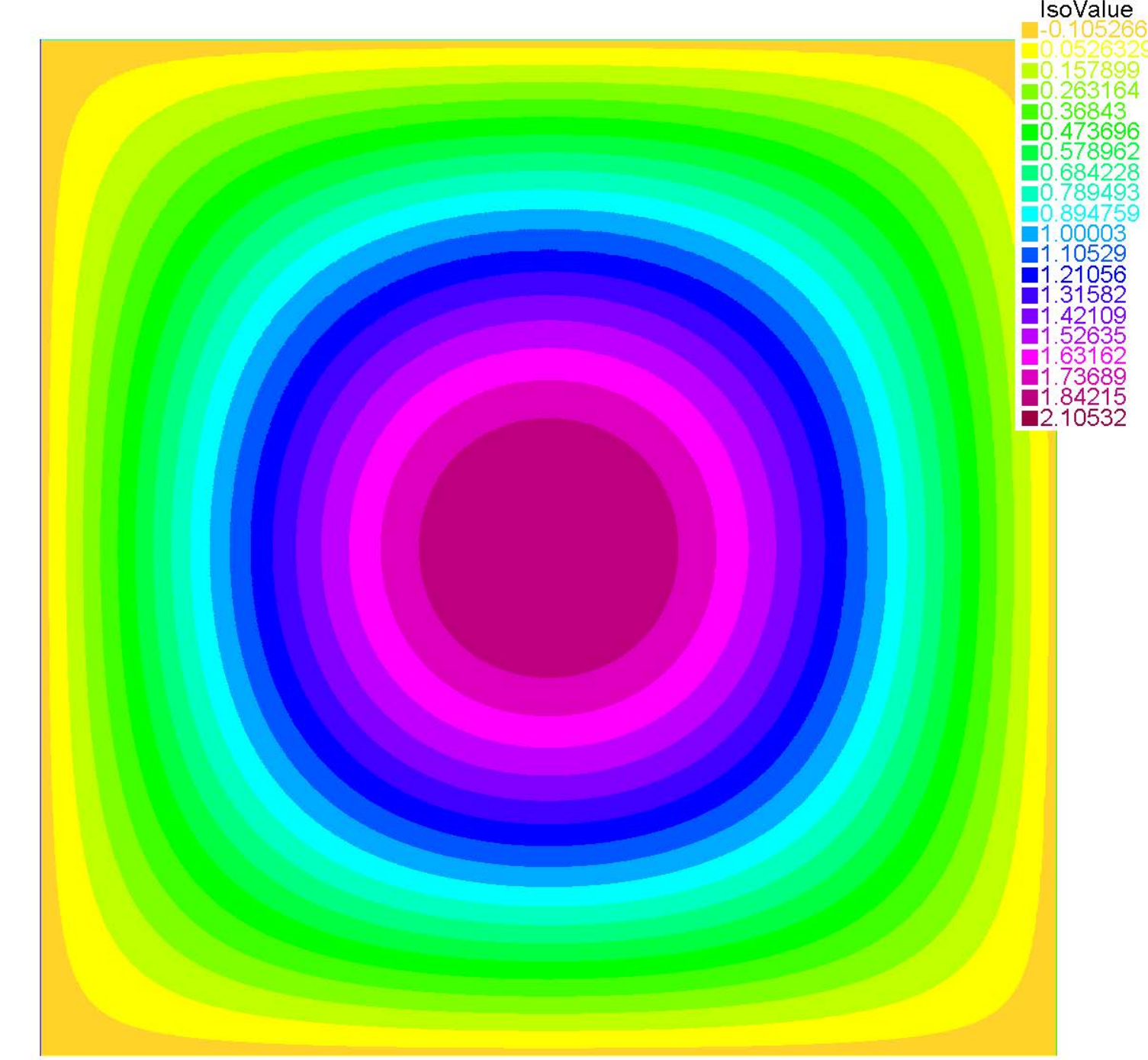}
\end{minipage}
\begin{minipage}[b]{0.5\textwidth}
\centering
\includegraphics[width=2.55in]{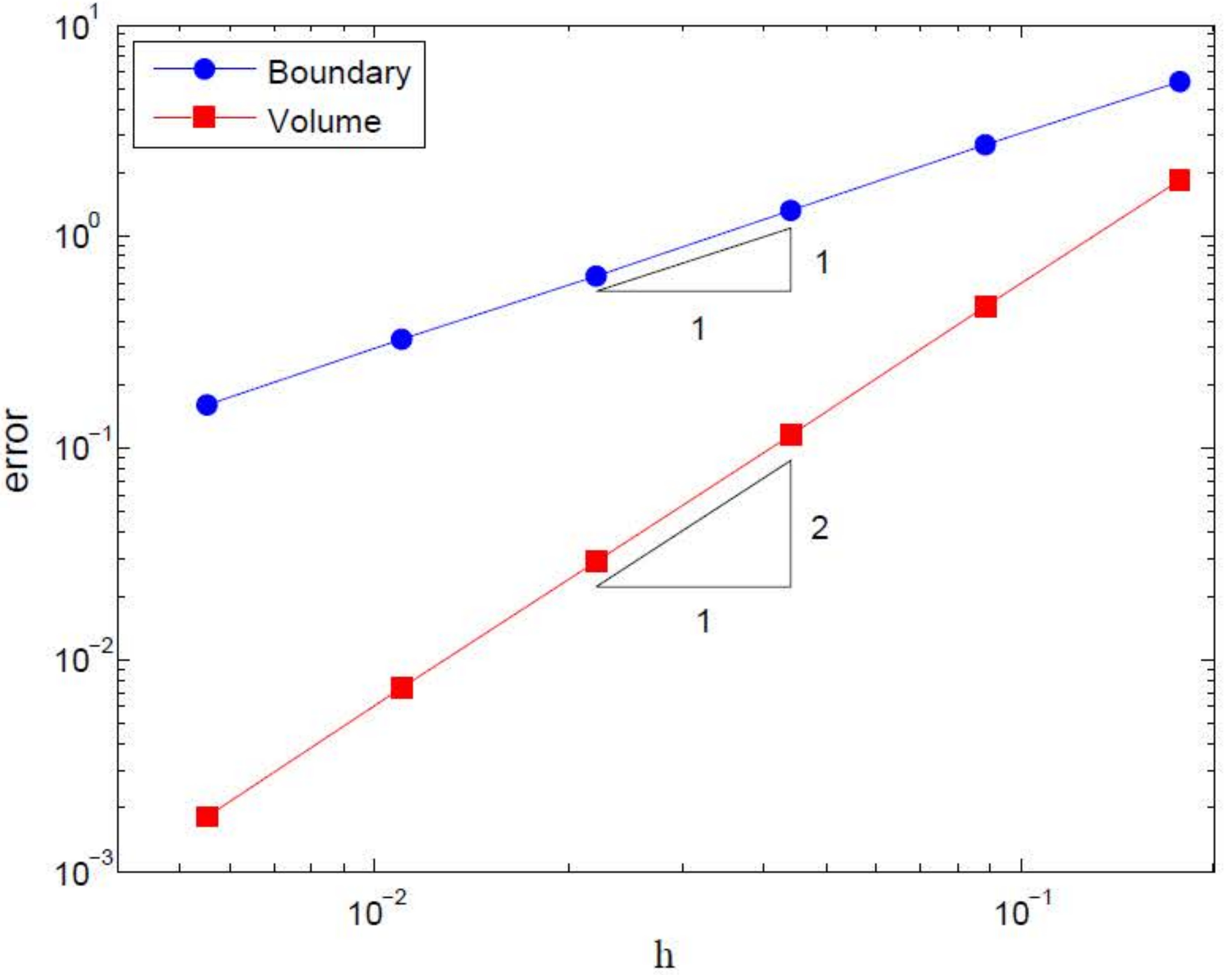}
\end{minipage}
\caption{Finite element approximation of eigenfunction on square (left) and the convergence history of approximate shape gradients (right).}
\label{figSq} 
\end{figure}
In Fig. \ref{figSq}, a linear finite element approximation of the first eigenfunction with $h=\sqrt{2}/256$ is illustrated, where the quadratic and linear convergence rates of approximate shape gradients in approximate operator norms agree well with the predicted results of Theorems \ref{Thm:domain}-\ref{Thm:Boundary}. For disk, the exact first eigenpair is
\begin{equation*}
\bigg(j^2_{0,1},\frac{1}{\sqrt{\pi}}\frac{1}{|J'_0(j_{0,1})|}J_0(j_{0,1}R)\bigg),
\end{equation*}
where $j_{0,1}$ is the first zero of the Bessel function $J_0$ and $R$ is the radial variable. The quasi-uniform triangulations based on uniform refinement are used. See Fig. \ref{figDisk} the computed eigenfunction under the finest mesh. The convergence of approximate shape gradients coincides well with theoretical results.
\begin{figure}[htb]
\begin{minipage}[b]{0.49\textwidth}
\centering
\includegraphics[width=2.3in]{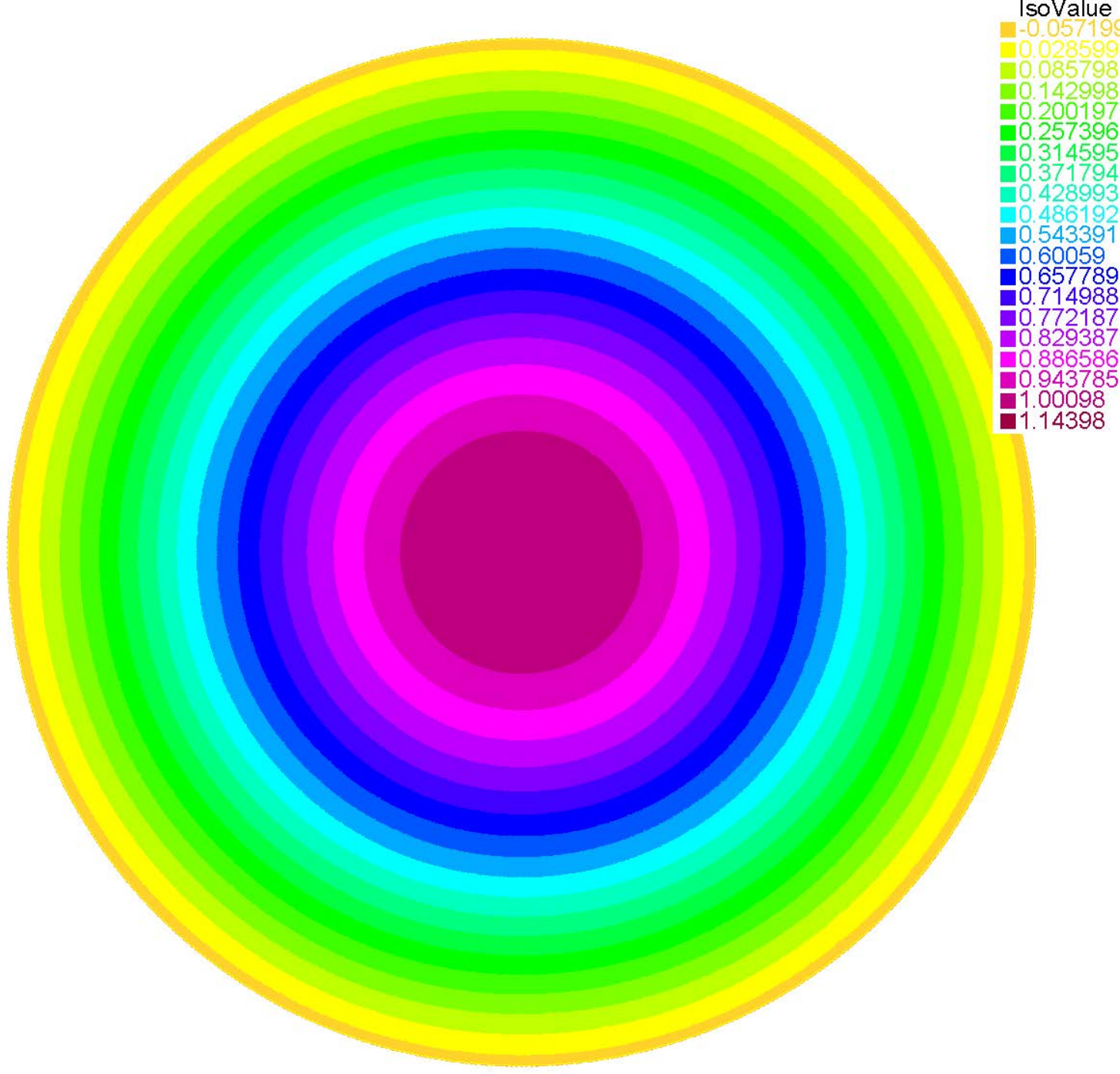}
\end{minipage}
\begin{minipage}[b]{0.51\textwidth}
\centering
\includegraphics[width=2.6in]{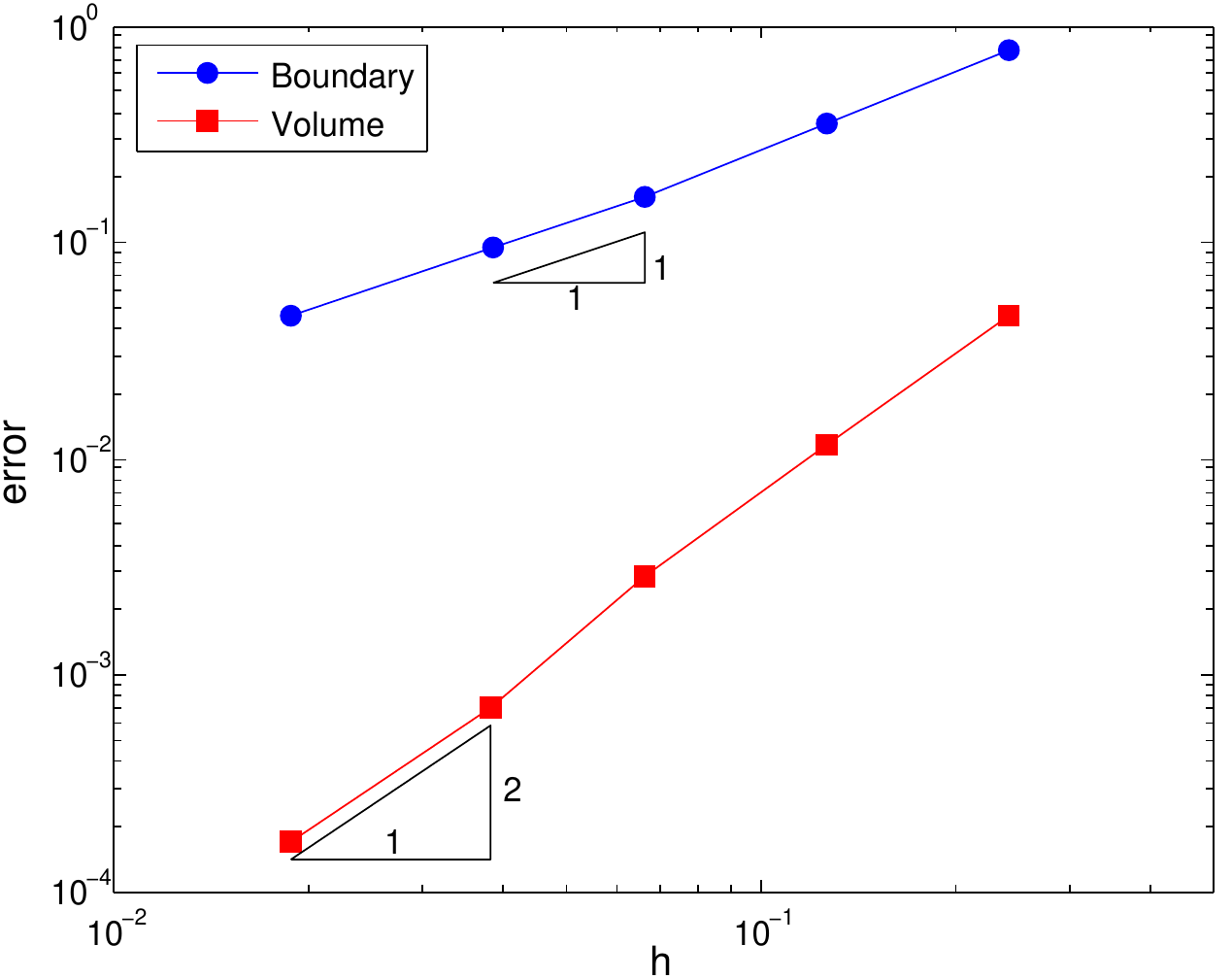}
\end{minipage}
\caption{Finite element approximation of Dirichlet eigenfunction on disk (left) and convergence history of approximate shape gradients (right).}
\label{figDisk} 
\end{figure}
\begin{figure}[htb]
\begin{minipage}[b]{0.49\textwidth}
\centering
\includegraphics[width=2.3in]{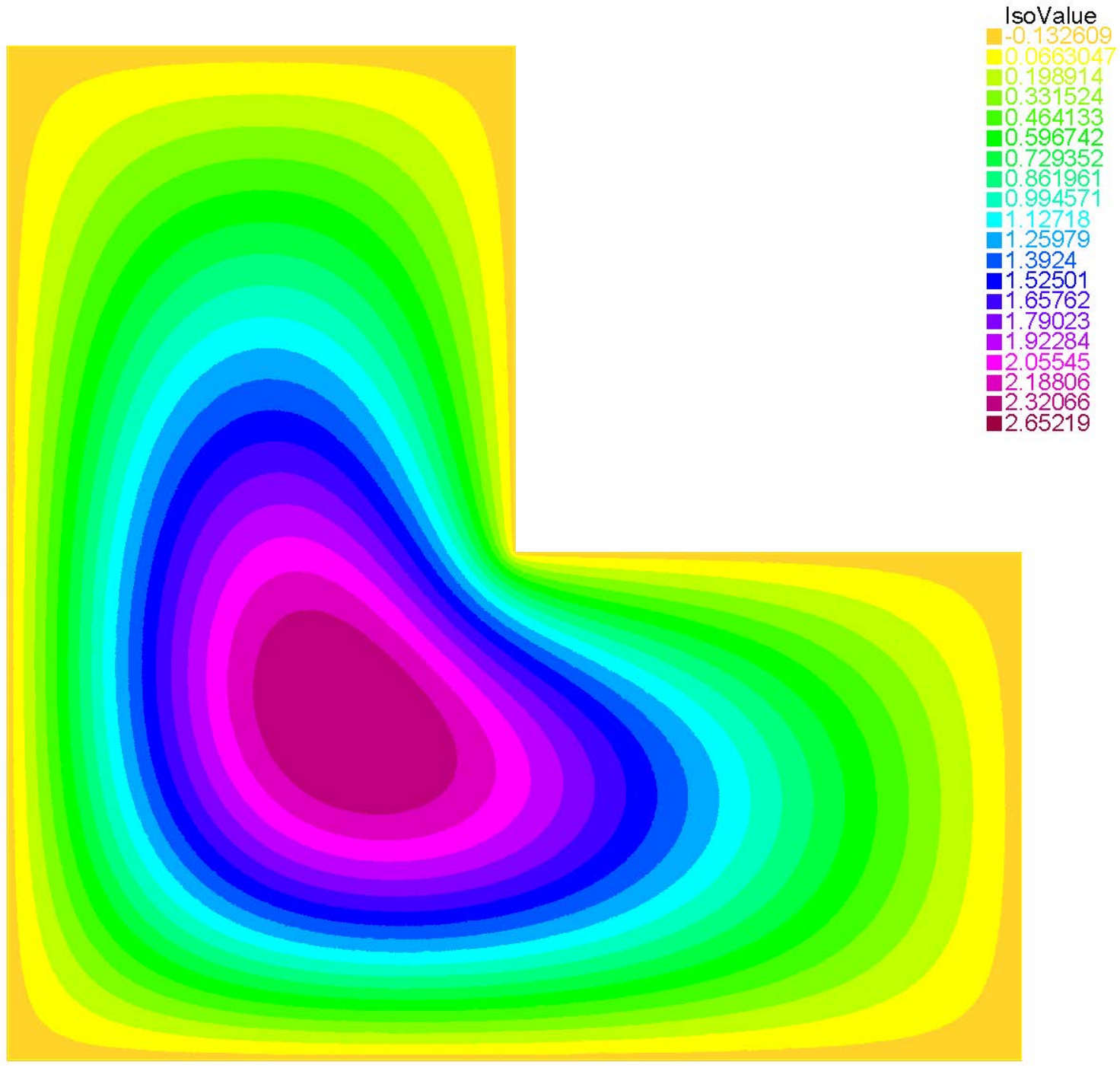}
\end{minipage}
\begin{minipage}[b]{0.51\textwidth}
\centering
\includegraphics[width=2.6in]{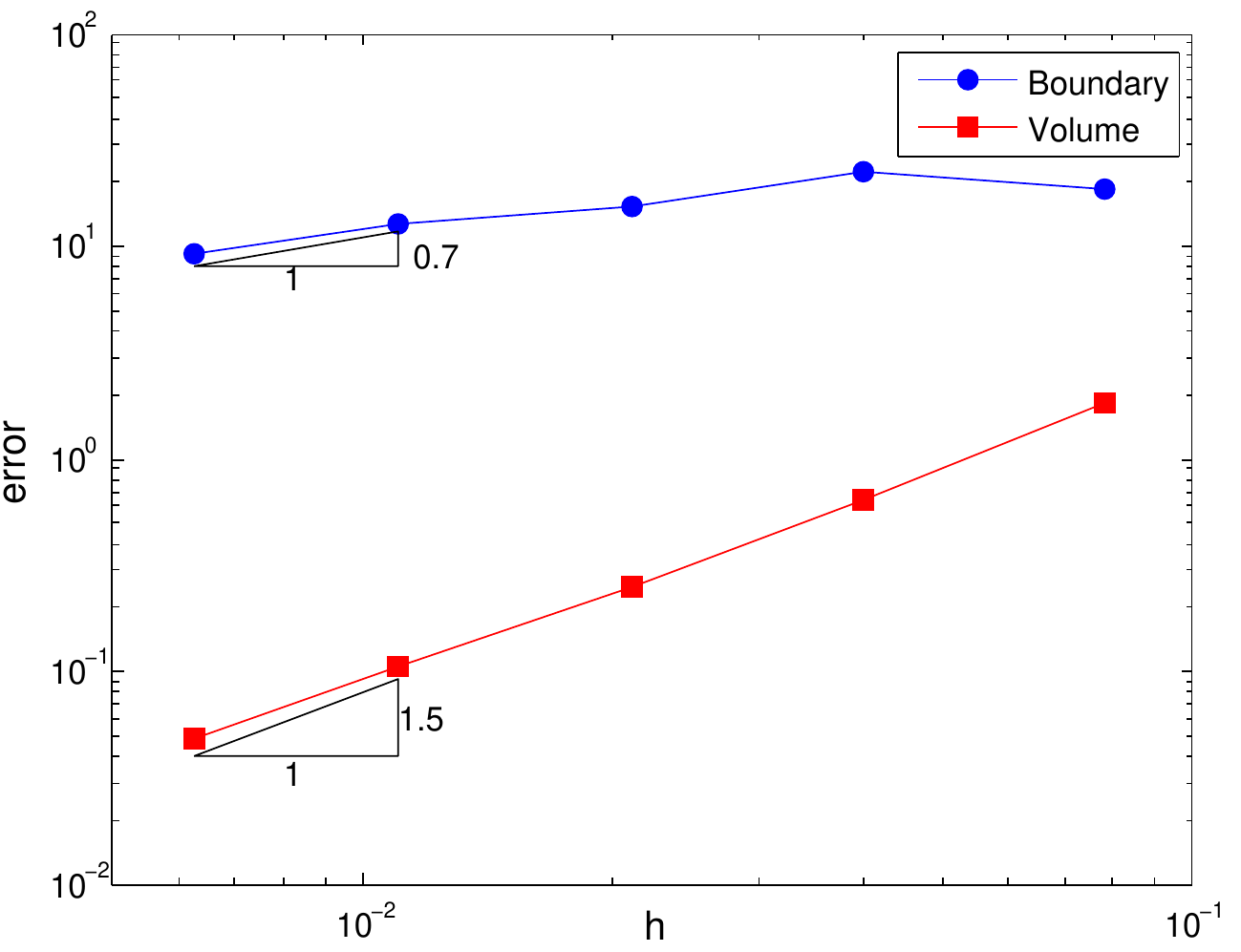}
\end{minipage}
\caption{Finite element approximation of Dirichlet eigenfunction on L-shaped domain (left) and convergence history of approximate shape gradients (right).}
\label{figLshape} 
\end{figure}
%
%

\begin{figure}[htb]
\subfigure[Square]{
\label{Aaa} 
\begin{minipage}[htb]{0.48\textwidth}
\centering
\includegraphics[width=2.4in]{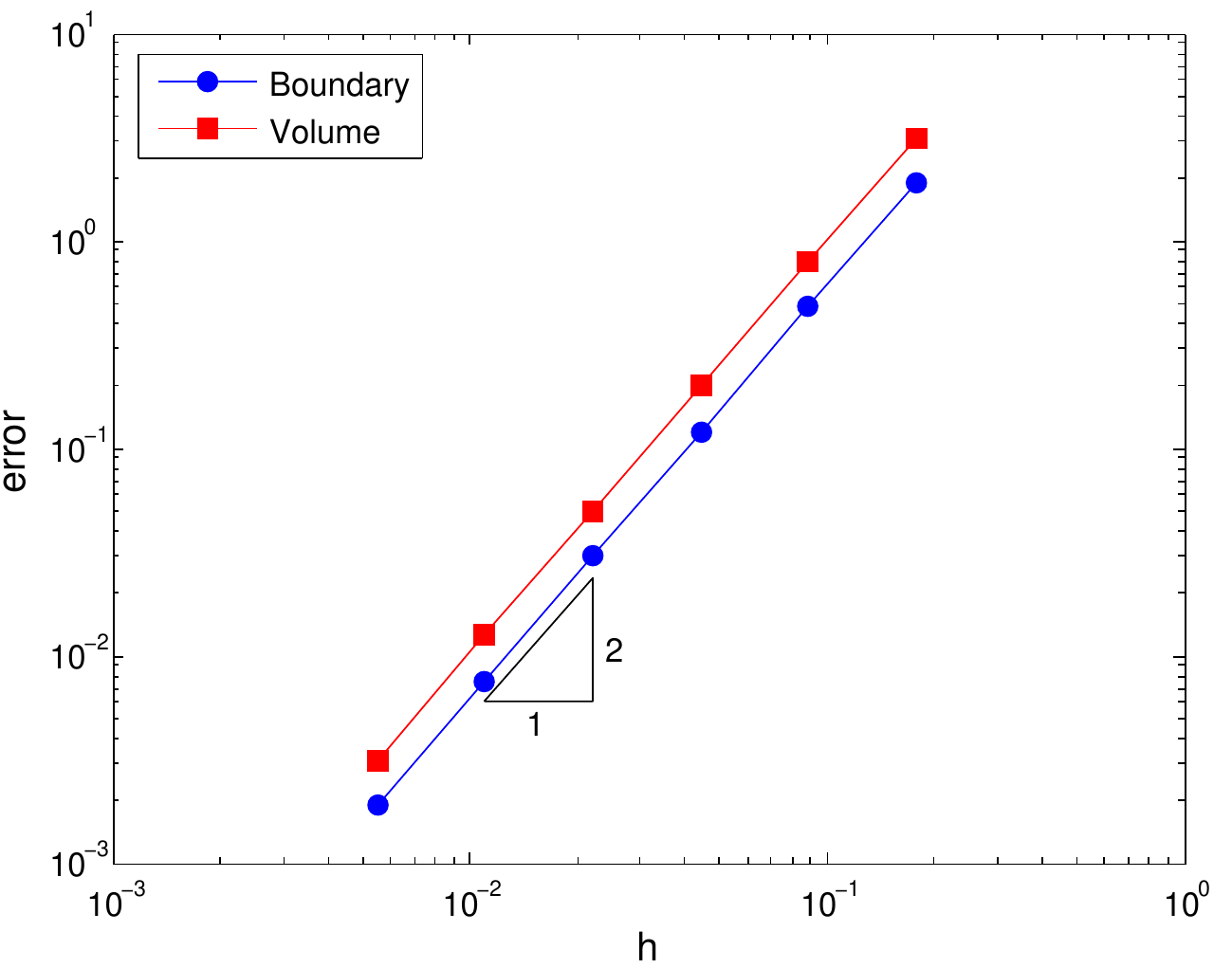}
\end{minipage}}
\hfill \subfigure[Disk]{
\label{Abb} 
\begin{minipage}[htb]{0.48\textwidth}
\centering
\includegraphics[width=2.4in]{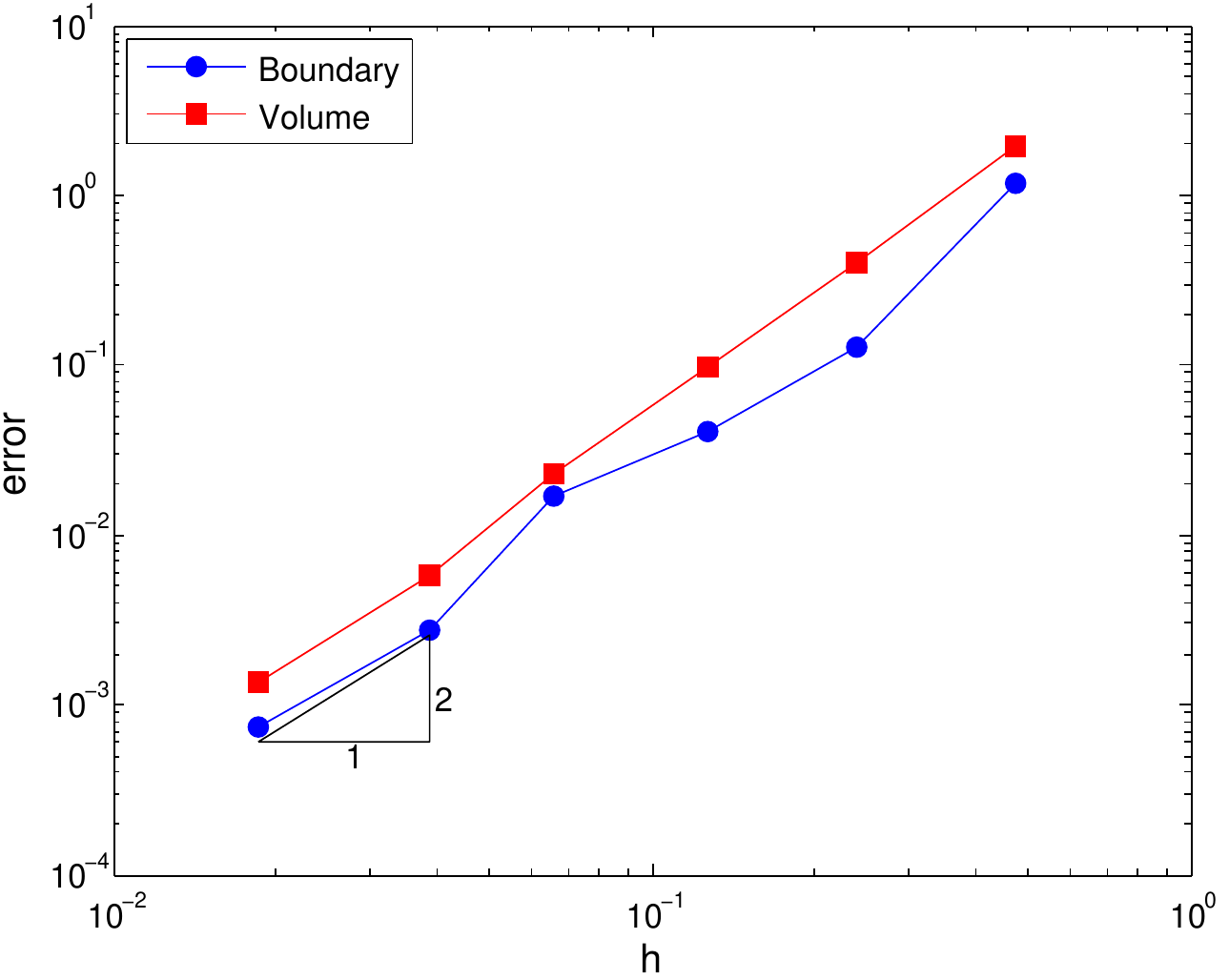}
\end{minipage}}
\hfill \subfigure[L-shape]{
\label{Acc} 
\begin{minipage}[htb]{0.9\textwidth}
\centering
\includegraphics[width=2.5in]{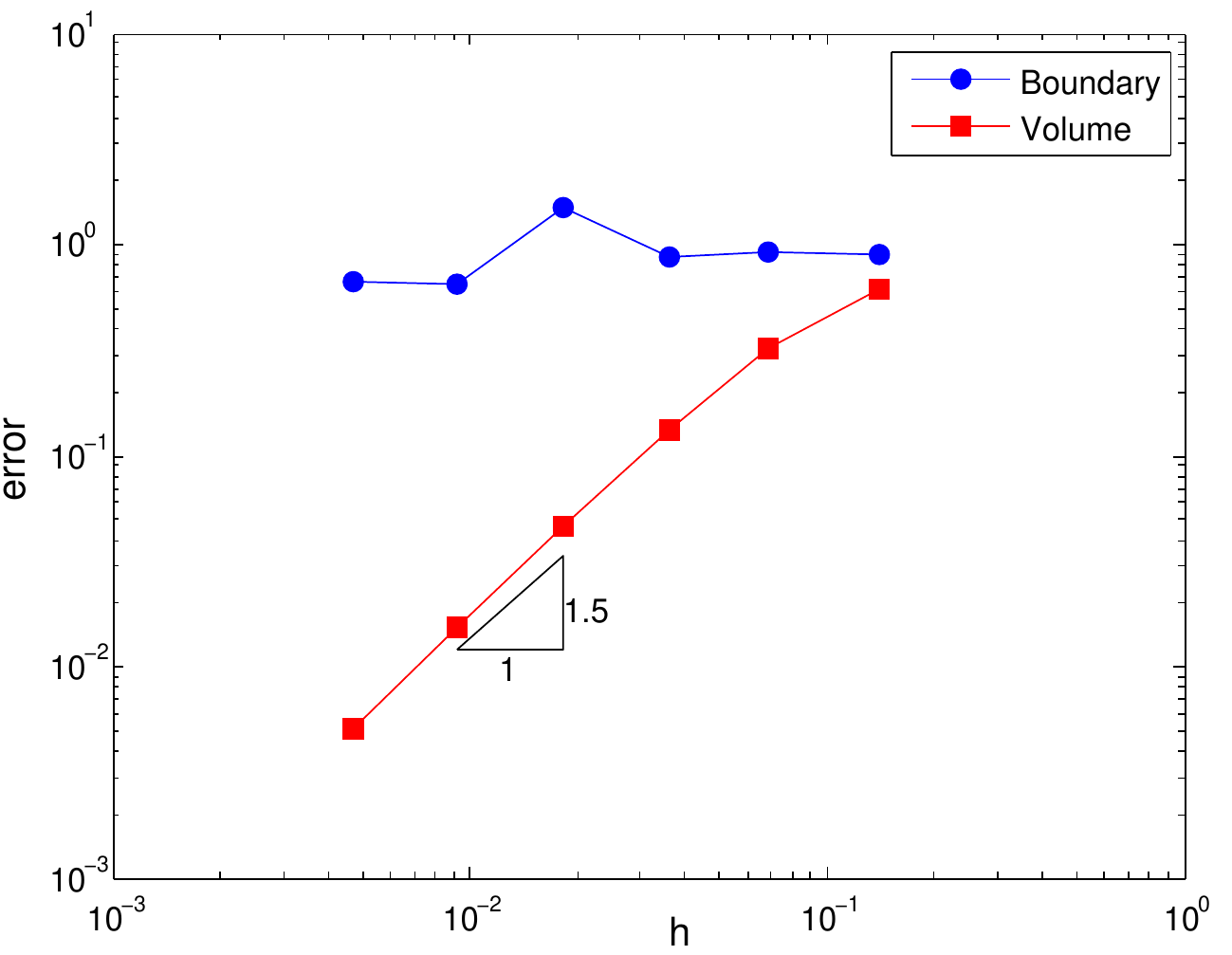}
\end{minipage}}
\caption{Convergence history of approximate shape gradients for Neumann boundary condition.}
\label{NeumannBoundCondLshape} 
\end{figure}
\begin{figure}[htb]
\subfigure[Square and Dirichlet]{
\label{AugMin1_0:a} 
\begin{minipage}[b]{0.48\textwidth}
\centering
\includegraphics[width=2.4in]{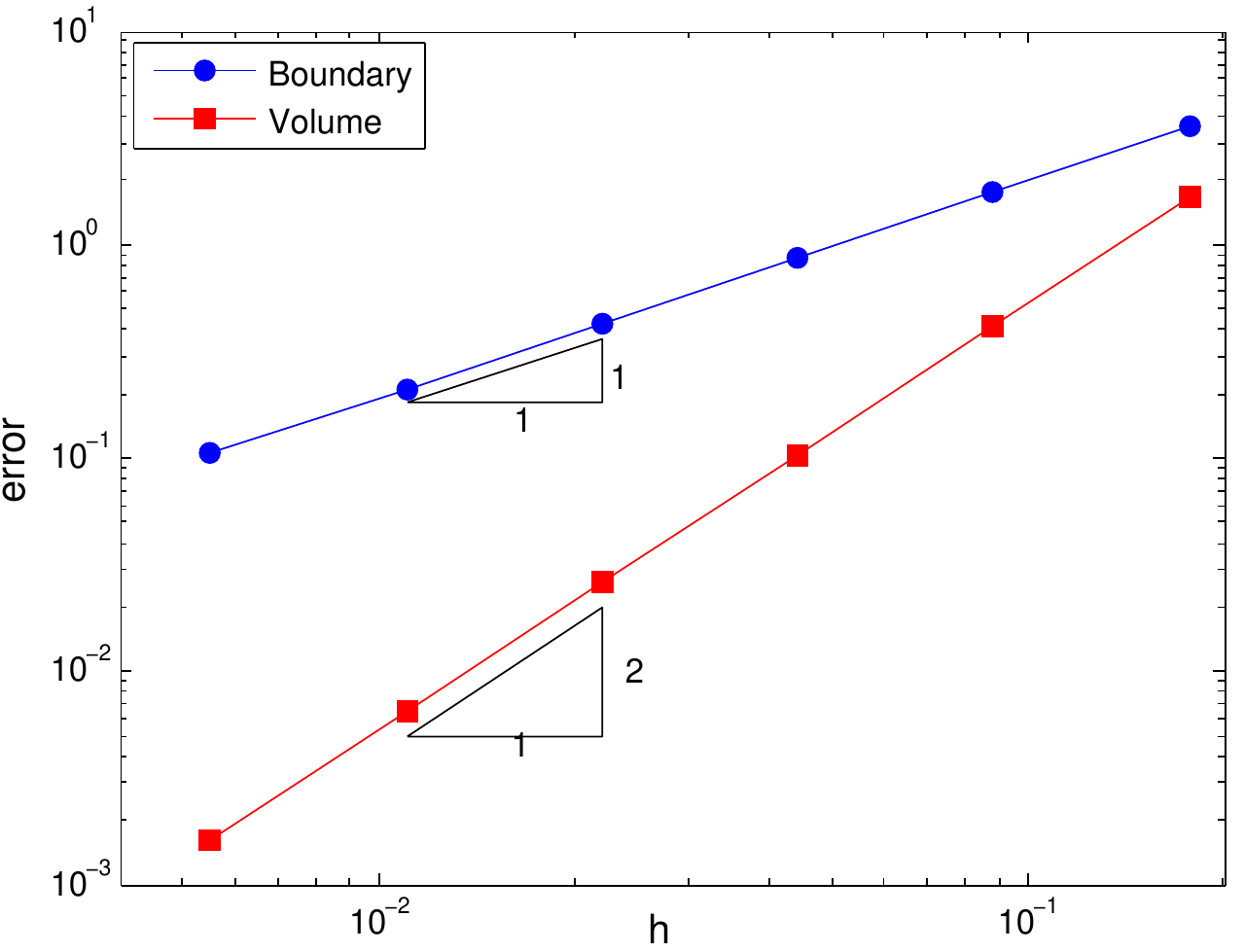}
\end{minipage}}
\hfill \subfigure[Square and Neumann]{
\label{AugMin1_40:b} 
\begin{minipage}[b]{0.48\textwidth}
\centering
\includegraphics[width=2.4in]{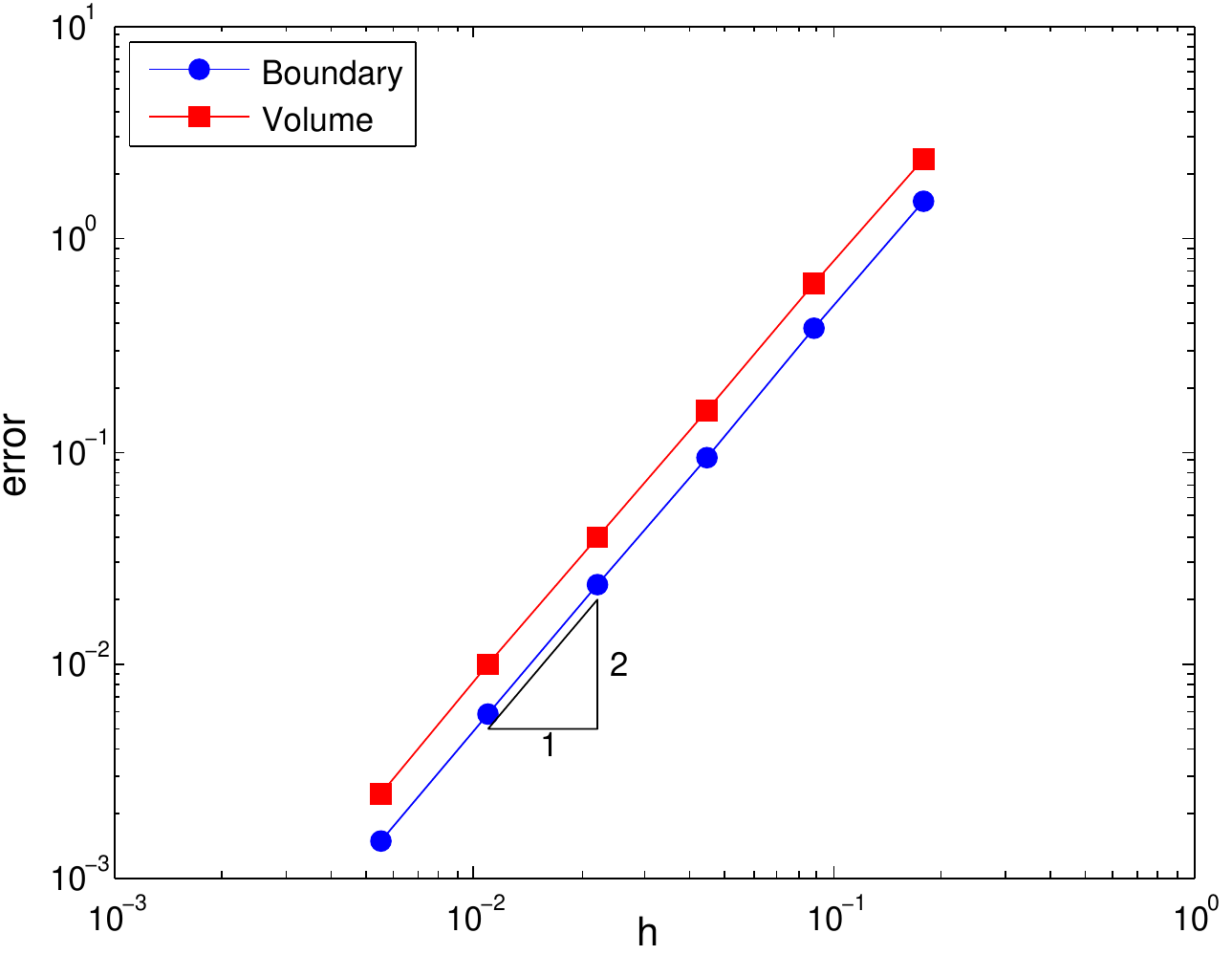}
\end{minipage}}
\hfill \subfigure[Disk and Dirichlet]{
\label{AugMin1_120:e} 
\begin{minipage}[b]{0.48\textwidth}
\centering
\includegraphics[width=2.45in]{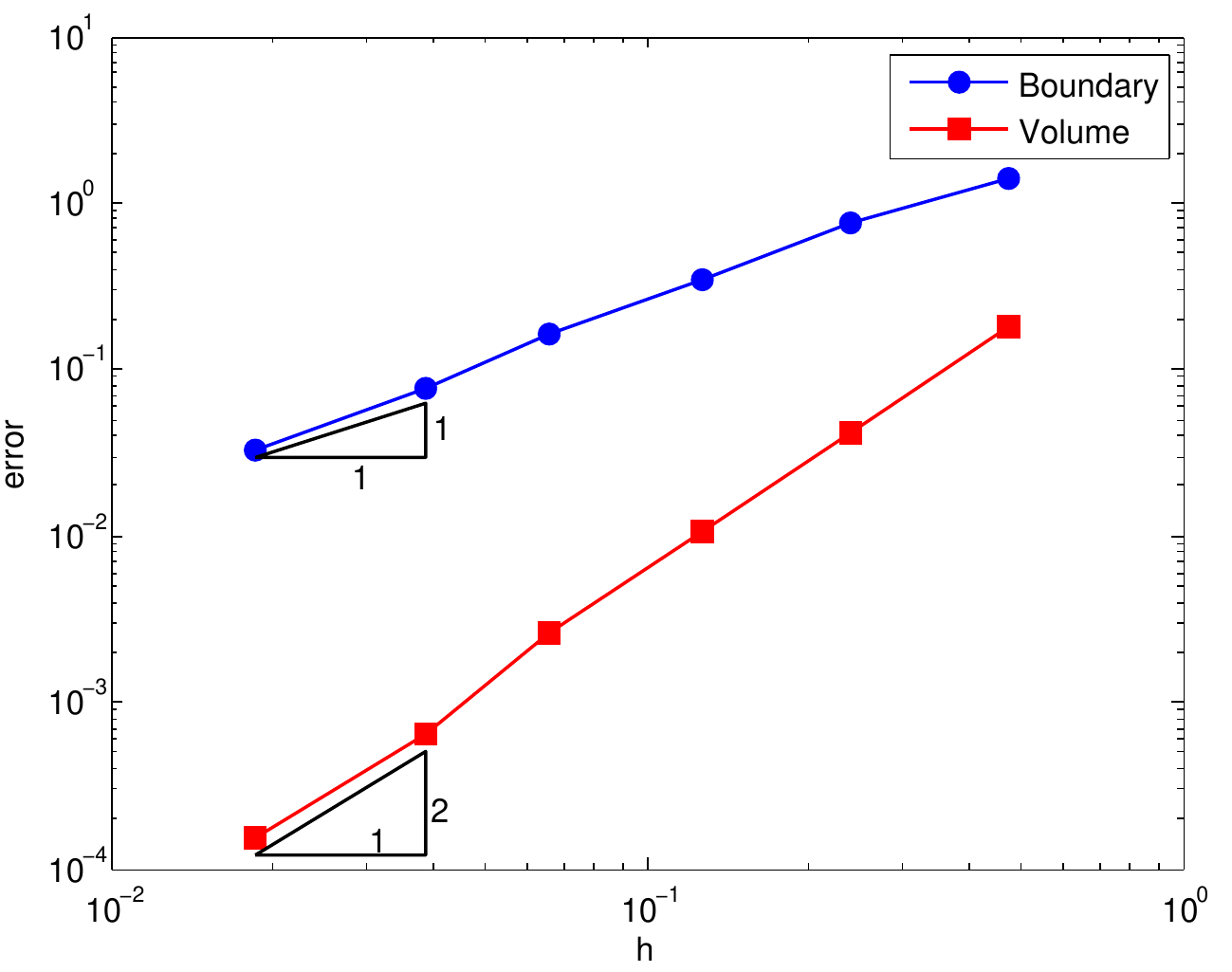}
\end{minipage}}
\hfill \subfigure[Disk and Neumann]{
\label{AugMin1_400:f} 
\begin{minipage}[b]{0.48\textwidth}
\centering
\includegraphics[width=2.4in]{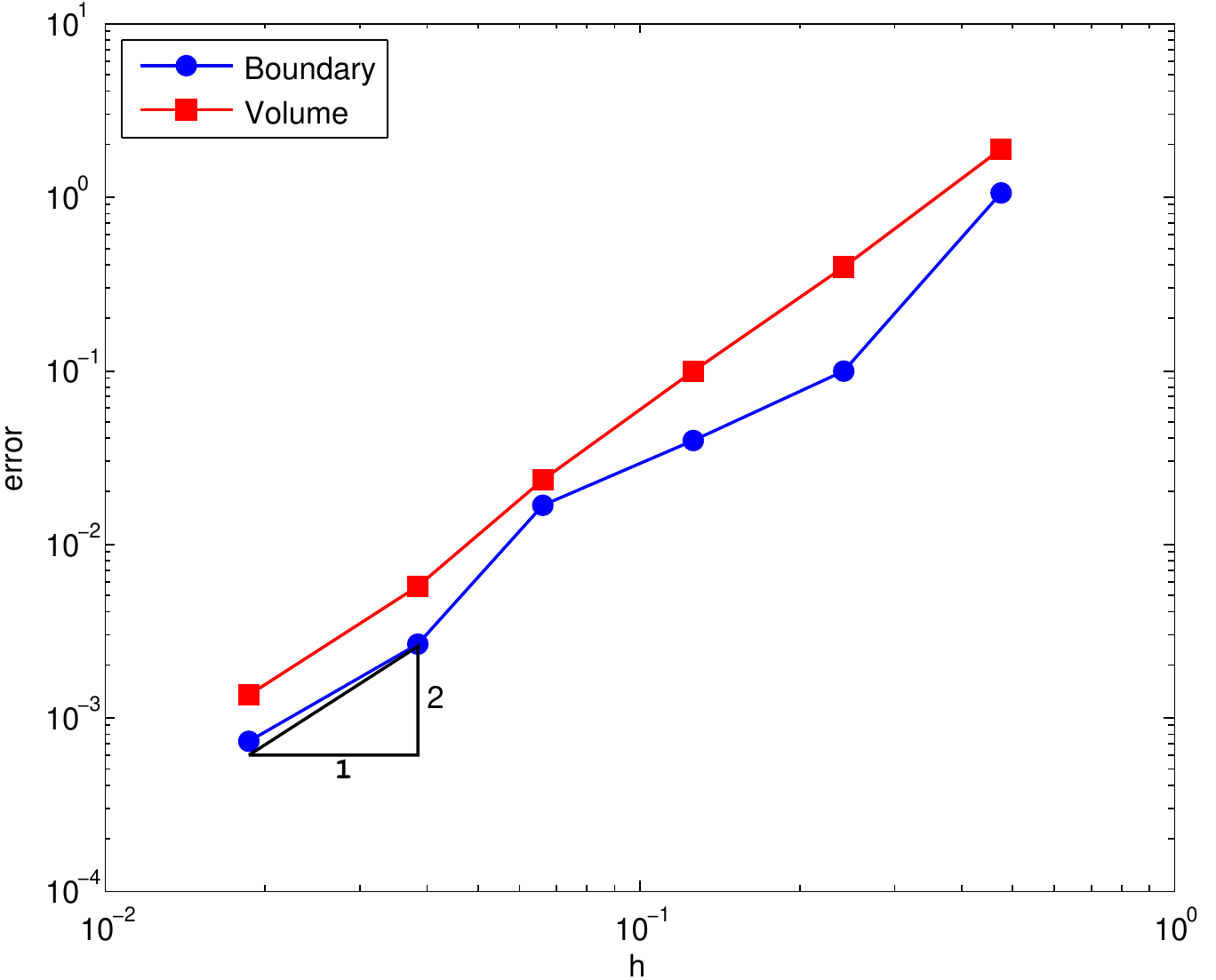}
\end{minipage}}
\caption{Convergence history of approximate shape gradients for the case of simple eigenvalues: $\gamma=2$.}
\label{NeumannPolyDegree2} 
\end{figure}
%
We use quasi-uniform meshes and compute the finite element solution on a fine mesh with 850523 degrees of freedom as the reference solution in Fig. \ref{figLshape}. The eigenfunction of L-shaped domain belongs to $H^{\frac{5}{3}-\epsilon}(\Omega)$ and thus the $H^2$-regularity for this eigenvalue problem does no hold. The approximate volume shape gradient is more accurate and converges faster than the boundary type.

\subsection{Neumann cases}
The first non-zero eigenvalue for square is multiple and we consider to optimize the second one. The exact eigenpair is
$(2\pi^2,2\cos{\pi x_1}\cos{\pi x_2})$.
For disk, we consider approximate the first exact non-zero eigenpair
\begin{equation*}
\bigg(j^{\prime,2}_{0,1},\frac{1}{\sqrt{\pi}}\frac{1}{|J_0(j^\prime_{0,1})|}J_0(j^\prime_{0,1}R)\bigg).
\end{equation*}
In Figs. \ref{AugMin1_40:b} and \ref{AugMin1_400:f}, the quadratic convergence rates of the boundary integral formula is unexpected, as has been observed in \cite{PaganiniThesis} for elliptic problem with Neumann boundary condition. For L-shape domain which does not guarantee the $H^2$-regularity of the eigenfunction, we show in Fig. \ref{NeumannBoundCondLshape} that the volume integral expression is superior to the boundary integral expression in terms of both accuracy and convergence rates. Finally, see Fig. \ref{NeumannPolyDegree2} the measured errors for the two boundary conditions in both square and disk when the multivariate polynomials of degree $\gamma=2$. The accuracy and converge rates agree well with those above when $\gamma=3$, which shows that the possible independent choice of $\gamma$ on computing the operator norm on the finite-dimensional subspace of multivariate polynomials vector fields.

\section{Conclusions}
We have performed comprehensive convergence analysis for Galerkin finite element approximations of the shape gradients for Dirichlet/Neumann type elliptic eigenvalue problems in shape optimization. The {\emph a priori} error estimates have been presented for the two types of approximate shape gradients in Eulerian derivatives. The convergence analysis for the volume type is performed under not restrictive assumptions on regularities of eigenfunctions and thus domains. For the Dirichlet case, theoretical analysis as well as numerical results have shown that the volume type formula converges faster and usually offers better accuracy. For the Neumann case, however, the boundary formulation is surprisingly competitive with the volume type.


\section*{Acknowledgements}
This work was supported in part by Science and Technology Commission of Shanghai Municipality (No. 18dz2271000) and the National Natural Science Foundation of China under grants 11201153 and 11571115.

\bibliographystyle{amsplain}

\appendix
\section{Eulerian derivatives of eigenvalues}
Closely following \cite{HenriBook} and Section 4.2 of Chapter 10 \cite{Defour} on the first Dirichlet eigenvalue, we give a heuristic and formal derivation only for the Dirichlet case for simplicity.
The results for the Neumann case can be derived similarly.
The variational formulation on $\Omega_t$ is to find $\lambda(\Omega_t)\in\mathbb{R}$, $0\neq u(\Omega_t)\in H^1_0(\Omega_t)$ such that
\begin{equation}\label{varFormt}
\int_{\Omega_t} \nabla u(\Omega_t)\cdot\nabla v{\rm d}x =\lambda(\Omega_t)\int_{\Omega_t} u(\Omega_t) v {\rm d}x\quad \forall v\in C_0^\infty(\Omega_t).
\end{equation}
From (\ref{varFormt}) and (\ref{varForm}), we have for all $\psi=v\circ T_t\in C_0^\infty(\Omega)$
\begin{equation*}\label{varFormTrans}
\int_{\Omega}\lim_{t\searrow 0}\big(B(t) \nabla (u(\Omega_t)\circ T_t)-\nabla u\big)\cdot\nabla \psi {\rm d}x =\int_{\Omega}\lim_{t\searrow 0}\big(\lambda(\Omega_t)\omega(t)u(\Omega_t) -\lambda(\Omega)u(\Omega)\big)\psi {\rm d}x,
\end{equation*}
where
\begin{equation*}
B(t) = \omega(t){\rm D}T_t^{-1}{\rm D}T_t^{-T}
\end{equation*}
with $\omega(t):={\rm det}({\rm D}T_t)$. The product rule for differentiation yields
\begin{equation}\label{VarPsi}
\int_{\Omega}\big(B'(0) \nabla u+\nabla \dot{u}\big)\cdot\nabla \psi {\rm d}x=\int_{\Omega}\big({\rm d}\lambda(\Omega;\mathcal{V})u+\lambda \dot{u}+\lambda u {\rm div}\mathcal{V}\big)\psi {\rm d}x,
\end{equation}
where
\begin{equation*}
B'(0)={\rm div}\mathcal{V} I -{\rm D}\mathcal{V}-{\rm D}\mathcal{V}^T
\end{equation*}
with $I\in \mathbb{R}^{d\times d}$ being the identity operator. Choosing $\psi=u$ in (\ref{VarPsi}), we have
\begin{equation}\label{key}
\int_{\Omega}\big(B'(0) \nabla u\cdot\nabla u+\nabla \dot{u}\cdot\nabla u\big) {\rm d}x=\int_{\Omega}\big({\rm d}\lambda(\Omega;\mathcal{V})u^2+\lambda u \dot{u}+\lambda u^2 {\rm div}\mathcal{V}\big) {\rm d}x.
\end{equation}
Then,
\begin{equation}\label{normalize}
\int_{\Omega_t} u^2(\Omega_t){\rm d}x=1
\end{equation}
corresponding to (\ref{contNorm}). Taking the derivative with respect to $t$ at $0$, we get
\begin{equation}\label{orth}
\int_{\Omega} 2u\dot{u}{\rm d}x+\int_{\Omega} u^2\, {\rm div} \mathcal{V}{\rm d}x= 0.
\end{equation}
By (\ref{contNorm}) and (\ref{orth}), (\ref{key}) implies that
\begin{equation}\label{djj}
{\rm d}\lambda(\Omega;\mathcal{V}) = \int_{\Omega} \big(B'(0)\nabla u\cdot\nabla u+\nabla u\cdot\nabla \dot{u}+\lambda u\dot{u}\big){\rm d}x.
\end{equation}
On the other hand, $\dot{u}=0$ on $\partial\Omega$ since $u(\Omega_t)$ vanishes $\partial\Omega_t$. Thus, $\dot{u}\in H^1_0(\Omega)$ if $u\in H^1_0(\Omega)$. Take $v=\dot{u}$ in (\ref{varForm}) and we obtain
\begin{equation}\label{VarformUmat}
\int_\Omega \nabla u\cdot \nabla \dot{u} {\rm d}x = \lambda \int_\Omega u \dot{u} {\rm d}x.
\end{equation}
A combination of (\ref{orth}), (\ref{djj}) and (\ref{VarformUmat}) yields the result (\ref{ThmDerEq}).

If, furthermore, $\Omega$ is convex or if it is of class $C^2$, then
$u\in H^2(\Omega)\cap H^1_0(\Omega)$ \cite{Babuska}. We can simplify the volume integral expression (\ref{ThmDerEq}) as a boundary integral expression (\ref{ShapeDer}).
Taking $v=u$ in the identity from vector calculus
\begin{equation*}
\mathcal{V}\cdot\nabla(\nabla u\cdot \nabla v) + \nabla u\cdot({\rm D}\mathcal{V}+{\rm D}\mathcal{V}^T)\nabla v = \nabla(\mathcal{V}\cdot\nabla u)\cdot\nabla v + \nabla u\cdot\nabla(\mathcal{V}\cdot\nabla v),
\end{equation*}
Eq. (\ref{ThmDerEq}) implies
\begin{equation}\label{ThmDerEqS}
\begin{aligned}
{\rm d}\lambda(\Omega;\mathcal{V}) &= \int_\Omega \Big( \mathcal{V}\cdot\nabla(|\nabla u|^2) - 2\nabla(\mathcal{V}\cdot\nabla u)\cdot\nabla u + {\rm div} \mathcal{V}(|\nabla u|^2-\lambda u^2)\Big){\rm d}x\\
&= \int_\Omega \Big( {\rm div}\big(|\nabla u|^2\mathcal{V}\big) - 2\nabla(\mathcal{V}\cdot\nabla u)\cdot\nabla u - \lambda{\rm div}\mathcal{V} u^2\Big){\rm d}x.
\end{aligned}
\end{equation}
By Green's theorem,
\begin{equation*}\label{ThmDerEqSs}
\begin{aligned}
 {\rm d}\lambda(\Omega;\mathcal{V}) &= \int_{\partial\Omega} \bigg( |\nabla u|^2\mathcal{V}_n {\rm d}s- 2\mathcal{V}\cdot\nabla u\frac{\partial u}{\partial n}\bigg) {\rm d}s + \int_\Omega\big(2\mathcal{V}\cdot\nabla u\Delta u - \lambda{\rm div}\mathcal{V} u^2\big){\rm d}x\\
&= -\int_{\partial\Omega}\bigg(\frac{\partial u}{\partial n}\bigg)^2\mathcal{V}_n {\rm d}s -\int_\Omega \big( 2\mathcal{V}\cdot \nabla u\lambda u + \lambda{\rm div}\mathcal{V} u^2\big){\rm d}x\\
&= -\int_{\partial\Omega}\bigg(\frac{\partial u}{\partial n}\bigg)^2\mathcal{V}_n {\rm d}s - \lambda\int_\Omega {\rm div}\big(u^2 \mathcal{V}\big){\rm d}x\\
&= -\int_{\partial\Omega}\bigg(\frac{\partial u}{\partial n}\bigg)^2\mathcal{V}_n {\rm d}s - \lambda\int_{\partial\Omega}u^2 \mathcal{V}_n{\rm d}s,
\end{aligned}
\end{equation*}
which implies (\ref{ShapeDer}) since $u$ vanishes on $\partial\Omega$.

\end{document}